\documentclass[12pt, a4paper]{amsart}
\usepackage{amsmath,amssymb,amsthm,booktabs,longtable}
\usepackage[loadonly]{enumitem}
\usepackage[all]{xy}
\usepackage{caption}
\usepackage{mathdots}
\usepackage{url}
\theoremstyle{plain}
\makeatletter 

\@addtoreset{table}{section}
\makeatother %
\newtheorem{theorem}{Theorem}[section]
\newtheorem{lemma}[theorem]{Lemma}
\newtheorem{corollary}[theorem]{Corollary}
\newtheorem{proposition}[theorem]{Proposition}
\newtheorem{definition}[theorem]{Definition}

\newtheorem{fact}[theorem]{Fact}
\newtheorem{example}[theorem]{Example}

\newtheorem{problem}[theorem]{Problem}
\newtheorem{remark}[theorem]{Remark}

\theoremstyle{definition}
\newtheorem*{claim}{Claim}

\newcommand{\Ad}{\mathop{\mathrm{Ad}}\nolimits}

\newcommand{\Hom}{\mathop{\mathrm{Hom}}\nolimits}

\newcommand{\rank}{\mathop{\mathrm{rank}}\nolimits}

\newcommand{\diag}{\mathop{\mathrm{diag}}\nolimits}
\newcommand{\even}{\mathrm{even}}

\newcommand{\Zariski}{\mathrm{Zar}}

\newcommand{\Q}{\mathbb{Q}}
\newcommand{\R}{\mathbb{R}}
\newcommand{\C}{\mathbb{C}}
\newcommand{\N}{\mathbb{N}}
\newcommand{\Z}{\mathbb{Z}}

\newcommand{\reductive}{real reductive group}

\newcommand{\trans}{{}^t\!}

\makeatletter
    
    \@addtoreset{equation}{section}
  \makeatother
\usepackage{float}%
\usepackage{enumerate}%
\usepackage[dvipdfm,
bookmarks=true,bookmarksnumbered=true,
bookmarksopen=true]{hyperref}
\makeatother
\begin{document}

\title[Zariski dense discontinuous surface groups]{Zariski dense discontinuous surface groups for reductive symmetric spaces}
\author{Kazuki Kannaka, Takayuki Okuda, Koichi Tojo}


\address[K.~Kannaka]{%
Faculty of Mathematics and Physics, Institute of Science and Engineering, Kanazawa University, Kakumamachi, Kanazawa, Ishikawa, 920-1192, JAPAN;
RIKEN Interdisciplinary Theoretical and Mathematical Sciences (\lowercase{i}THEMS), 
Wako, Saitama 351-0198, Japan.
        }
\email{kannaka@se.kanazawa-u.ac.jp}

\address[T.~Okuda]{%
	Graduate School of Advanced Science and Engineering, Hiroshima University, 
    1-3-1 Kagamiyama, Higashi-Hiroshima City, Hiroshima, 739-8526, Japan.
        }
\email{okudatak@hiroshima-u.ac.jp}

\address[K.~Tojo]{%
	RIKEN Center for Advanced Intelligence Project, Nihonbashi 1-chome Mitsui Building, 15th Floor, 1-4-1 Nihonbashi, Chuo-ku, Tokyo 103-0027, Japan
        }
\email{koichi.tojo@riken.jp}

\begin{abstract}
Let $G/H$ be a homogeneous space of reductive type with non-compact $H$.
The study of deformations of discontinuous groups for $G/H$ was initiated by T.~Kobayashi. 
In this paper, we show that a standard discontinuous group $\Gamma$ admits 
a non-standard small deformation as a discontinuous group for $G/H$
if $\Gamma$ is isomorphic to a surface group of high genus
and its Zariski closure is locally isomorphic to $SL(2,\mathbb{R})$. 
Furthermore, we also prove that 
if $G/H$ is a symmetric space and admits some non virtually abelian discontinuous groups, 
then $G$ contains a Zariski-dense discrete surface subgroup of high genus
acting properly discontinuously on $G/H$.
As a key part of our proofs, 
we show that
for a discrete surface subgroup $\Gamma$ of high genus contained in a reductive group $G$, 
if the Zariski closure of $\Gamma$ is locally isomorphic to $SL(2,\mathbb{R})$,
then $\Gamma$ admits a small deformation in $G$ whose Zariski closure is a reductive subgroup 
of the same real rank as $G$.
\end{abstract}

\date{}
\maketitle
\dedicatory{Dedicated to Professor Toshiyuki Kobayashi with admiration for the plentiful and exciting mathematics he pioneered.}

\setcounter{tocdepth}{1}
\tableofcontents

\section{Introduction and main results}
\label{section:introduction}

Let $G/H$ be a homogeneous space of reductive type with a proper $SL(2,\R)$-action.
We consider a discrete subgroup $\Gamma$ of $SL(2,\R)$ isomorphic to the fundamental group $\pi_1(\Sigma)$ of a closed oriented surface $\Sigma$  of high genus.
Then the $\Gamma$-action on $G/H$ is properly discontinuous. 
In this paper, 
we show that such a $\Gamma$ always admits a non-standard small deformation in $G$ as a discontinuous group for $G/H$.
Furthermore, we also prove that 
if $G/H$ is a symmetric space, 
then $G$ contains a Zariski-dense discrete subgroup isomorphic to $\pi_1(\Sigma)$
which acts properly discontinuously on $G/H$.

We motivate our work in the following problem on Clifford--Klein forms:
\begin{problem}
\label{problem:intro-main}
Let $G/H$ be a homogeneous space of reductive type.
What types of discrete subgroups of $G$ 
can appear as discontinuous groups for $G/H$?
\end{problem}

We recall some basic notions. Let $\Gamma$ be a discrete subgroup of $G$.
If the $\Gamma$-action on $G/H$ is properly discontinuous and free, 
the quotient space $\Gamma\backslash G/H$ carries a unique $C^{\infty}$-manifold structure such that  
the quotient map $G/H\rightarrow \Gamma\backslash G/H$ is a covering of
$C^{\infty}$ class. Thus $\Gamma\backslash G/H$ becomes a $(G,G/H)$-manifold in the sense of Ehresmann and Thurston.
Following Kobayashi \cite{Kobayashi-unlimit}, we call $\Gamma$ a \textit{discontinuous group for} $G/H$ 
and $\Gamma\backslash G/H$  a \textit{Clifford--Klein form} of $G/H$.
We should note that the assumption of proper discontinuity imposes a strong constraint when $H$ is non-compact.
For instance, Kobayashi proved in \cite{Kobayashi89} that no infinite discrete subgroups of $G$ act properly discontinuously on $G/H$
if and only if $\rank_{\R}(G) = \rank_{\R}(H)$ (the Calabi--Markus phenomenon). 
The systematic study of discontinuous groups for general $G/H$ with non-compact $H$
was initiated by Kobayashi's pioneering works \cite{Kobayashi89,Kobayashi1992necessary} in 1980's, and has been developed by many researchers in various approaches.

Let us also recall from \cite[Definition 1.4]{KasselKobayashi16} that 
a discontinuous group $\Gamma$ for $G/H$ is called \emph{standard} if $\Gamma$ is contained in a reductive subgroup of $G$ acting properly on $G/H$.
It should be noted that any discontinuous group for $G/H$ with compact $H$ is standard since the $G$-action on $G/H$ is proper. 
Even for non-compact $H$, 
based on Kobayashi's properness criterion \cite{Kobayashi89}, 
standard discontinuous groups isomorphic to uniform lattices of reductive Lie groups have been much investigated
(for example \cite{Kobayashi89,Kobayashi96, KobayashiYoshino05,Kulkarni1981proper,Okuda13,Tojo2019Classification}). 

In this paper we focus on the following problem:
\begin{problem}\label{problem:Gamma_nonstandard_Zariskidense}
Let $G/H$ be a homogeneous space of reductive type with non-compact $H$ and $\Gamma_{0}$ an abstract discrete group. 
\begin{enumerate}
        \item\label{item:problem:Gamma_nonstandard} 
        Does there exist a 
        non-standard discontinuous group $\Gamma$ for $G/H$ 
        which is isomorphic to $\Gamma_{0}$ as a group?
        \item \label{item:problem:Gamma_Zariskidense}
        If it exists, can we choose such a $\Gamma$ to be Zariski-dense in $G$?
\end{enumerate}
\end{problem}

As mentioned in more detail in Sections~\ref{intro:non-standard} and~\ref{section:intro-zariski-dense}, 
Problem~\ref{problem:Gamma_nonstandard_Zariskidense} has been investigated in the following two types of settings:
\begin{itemize}
\item (\cite{Ghys95,Goldmannonstandard,Kassel12,Kobayashi98}, see Section~\ref{intro:non-standard})
$G/H$ is specific, $\Gamma_{0}$ depends on $G/H$ and $\Gamma\backslash G/H$ is compact.
\item (\cite{Benoist96}, see Section~\ref{section:intro-zariski-dense})
$G/H$ is a general homogeneous space of reductive type and $\Gamma_{0}$ is independent of $G/H$, but $\Gamma\backslash G/H$ is no longer compact.
\end{itemize}
In this paper, we study Problem~\ref{problem:Gamma_nonstandard_Zariskidense} in the second type of the setting where 
$G/H$ is general and $\Gamma_{0}$ is the surface group $\pi_{1}(\Sigma)$.

The study of properly discontinuous actions of surface groups on homogeneous spaces has long been an established topic. Recent developments in this field include Danciger--Gu\'{e}ritaud--Kassel's construction of surface groups acting properly on affine spaces 
\cite{DancigerGueritaudKassel20}.

\subsection{Zariski-dense small deformation of Fuchsian surface groups}

To study Problem~\ref{problem:Gamma_nonstandard_Zariskidense} for surface groups, we temporarily set aside the issue of proper discontinuity to consider the following problem:

\begin{problem}
    \label{prob:Zariski-dense}
    Let $\rho$ be a Lie group homomorphism from $SL(2,\R)$ to
    a real reductive Lie group $G$
    and $\Gamma_{g}$ a discrete subgroup of $SL(2,\R)$ isomorphic to $\pi_{1}(\Sigma_{g})$ with genus $g\geq 2$.
    Determine when the surface group representation $\rho|_{\Gamma_{g}}\colon \Gamma_{g}\rightarrow G$
    admits a Zariski-dense small deformation. 
\end{problem}

As seen in works of Toledo~\cite{Toledo89}, Bradlow--Garc\'{i}a-Prada--Gothen~\cite{BradlowPradaGothen2003}, Burger--Iozzi--Wienhard~\cite{BurgerIozziWienhard10} and among others (see also references in the introduction of \cite{BurgerIozziWienhard10}), 
Problem~\ref{prob:Zariski-dense} is nontrivial.
In the context of studying the deformation of Fuchsian surface groups, a recent construction of quasi-Fuchsian surface subgroups in cocompact discrete subgroups of semisimple Lie groups by Kahn--Labourie--Mozes~\cite{kahn2024surface} is also noteworthy.

Here let us recall from \cite{BradlowPradaGothen2003,BurgerIozziWienhard10} an example of a Fuchsian surface group representation which can not 
be deformed into a Zariski-dense representation. 
Let $\rho\colon SL(2,\mathbb{R})\rightarrow SU(p,q)$ $(p\geq q)$
be the following Lie group homomorphism:
\[
  SL(2,\mathbb{R}) \xrightarrow{\simeq} SU(1,1) \xrightarrow{\text{diagonal}} SU(1,1)^{q} 
\xrightarrow{\text{standard}} SU(q,q) \xrightarrow{\text{standard}} SU(p,q).  
\]
In this case, the following rigidity theorem holds:
\begin{fact}[\cite{BradlowPradaGothen2003,BurgerIozziWienhard10}]
    \label{fact:BIW}
    In the setting above, 
    any small deformation $\varphi\in \Hom(\Gamma_{g},SU(p,q))$ of the representation $\rho|_{\Gamma_{g}}$
    can be conjugated into $S(U(q,q)\times U(p-q))$. 
    In particular,  $\rho|_{\Gamma_{g}}$ can not be 
    deformed into a Zariski-dense representation 
    if $p>q$.
\end{fact}

More generally, Burger--Iozzi--Wienhard~\cite{BurgerIozziWienhard10}
proved that any maximal representation of a surface group 
to a reductive group $G$ of Hermitian type never admits a Zariski-dense small deformation 
if $G$ is of non-tube type.

Burger--Iozzi--Wienhard~\cite{BurgerIozziWienhard10}
also showed that 
if $G$ is Hermitian of tube type 
and $\rho \colon SL(2,\R) \rightarrow G$ corresponds to a diagonal disk, 
then the surface group representation $\rho|_{\Gamma_{g}}$ to $G$ 
can be deformed into a Zariski-dense representation.
Furthermore, 
for classical simple $G$,
if $G$ is Hermitian of tube type or non-Hermitian, 
Kim--Pansu showed in \cite{KIMPAN15} that,  
any surface group representation to $G$ with sufficiently large genus 
can be deformed into a Zariski-dense representation.
However, when including exceptional $G$ of non-Hermitian type, it seems that no simple criterion for Problem~\ref{prob:Zariski-dense} is known, to the best knowledge of the authors.

In this paper, we give a combinatorial sufficient condition (Corollary~\ref{maincorollary:even_ZariskiDense}) for Problem~\ref{prob:Zariski-dense} in terms of nilpotent orbits in complex reductive Lie algebras.
Our criterion is applicable in the case where $G$ is a general real reductive algebraic group.
It should also be emphasized that our proof provides an explicit construction of Zariski-dense small deformations of $\rho|_{\Gamma_{g}}$.

To state our criterion, we introduce a few notation (see Section~\ref{subsection:terminology} for the notation on real reductive Lie groups).
\begin{definition}
\label{intro_def:rho-G-g}
    For each Lie group homomorphism
    $\rho\colon SL(2,\mathbb{R})\rightarrow G$,
    we put 
    \[
    \sigma\equiv \sigma(\rho):=\mathrm{exp}(\pi\sqrt{-1}d\rho(\begin{pmatrix}1 & 0 \\
    0& -1\end{pmatrix}))\in \mathbf{G}(\C).
    \]
    Let us define an algebraic subgroup $\mathbf{G}_{\even}^{\rho}$ of $\mathbf{G}$
    as the Zariski identity component of the centralizer $\{x\in \mathbf{G}\mid \sigma x=x\sigma\}$.
    Since $\mathbf{G}_{\even}^{\rho}$ is defined over $\R$ as we see in Lemma~\ref{lemma:rho-G-g}, 
    we define an open subgroup $G_{\even}^{\rho}$ of $\mathbf{G}_{\even}^{\rho}(\R)$ as
    \[
    G_{\even}^{\rho}:=G\cap \mathbf{G}_{\even}^{\rho}(\R).
    \]
\end{definition}
We note that 
$G^{\rho}_{\even}$ has the same real rank as $G$ (see Lemma~\ref{lemma:grhoev_fulrank}).
For instance in the setting of Fact~\ref{fact:BIW},
$G^{\rho}_{\even}$ coincides with $S(U(q,q)\times U(p-q))$.

Here is our main result for Problem~\ref{prob:Zariski-dense}.
In our proof, we explicitly construct the following small deformation 
by using an idea of bending construction (see Lemma~\ref{lemma:deform-of-surface-group}):
\begin{theorem}[see also Theorem~\ref{theorem:deformation-zariski-closure}]
\label{maintheorem:deformation-zariski-closure}
In the setting of Problem~\ref{prob:Zariski-dense},
if the genus $g$ is sufficiently large, then one can construct 
a small deformation $\rho'$ of $\rho|_{\Gamma_{g}}$ such that 
the Zariski closure of $\rho'(\Gamma_{g})$ coincides with $G^{\rho}_{\even}$.
\end{theorem}

We recall that 
by the Jacobson--Morozov theorem, 
the set of Lie group homomorphisms $\rho \colon SL(2,\R) \rightarrow G$ up to conjugations by the identity component of $G$,  
has a bijective correspondence to the set of real nilpotent (adjoint) orbits in the Lie algebra $\mathfrak{g}$ of $G$ (see \cite[Chapter 9]{CollingwoodMcGovern} for the details).
In this paper, a real nilpotent orbit $\mathcal{O}$ in $\mathfrak{g}$ is called \emph{even} if its complexification is a even complex nilpotent orbit in $\mathfrak{g} \otimes \C$ in the sense of \cite[Chapter 3.8]{CollingwoodMcGovern}.
Then as a corollary to Theorem \ref{maintheorem:deformation-zariski-closure}, we have the following:

\begin{corollary}\label{maincorollary:even_ZariskiDense}
    In the setting of Problem~\ref{prob:Zariski-dense},
    assume that $\rho \colon SL(2,\R) \rightarrow G$ is corresponding to an even real nilpotent orbit in the sense above and that the genus $g$ is sufficiently large.
    Then the surface group representation $\rho|_{\Gamma_{g}}$ can be deformed to 
    a Zariski-dense representation.
\end{corollary}

It should be remarked that classifications of 
Lie group homomorphisms $\rho \colon SL(2,\R) \rightarrow G$ corresponding to even nilpotent orbits 
can be read off in \cite[Chapter 9]{CollingwoodMcGovern}. 
Therefore, we can obtain abundant examples of $\rho$ to which Corollary \ref{maincorollary:even_ZariskiDense} is applicable.
For example, 
in \cite[Chapter 9.6]{CollingwoodMcGovern}, 
in the table for $G = E_{6(2)}$, which is neither split nor
of Hermitian type,  
$18$ types of such $\rho$ can be found as labeled $0$ (which is corresponding to the trivial homomorphism), $6$, $7$, $8$, $11$, $20$, $21$, $22$, $23$, $24$, $25$, $26$, $32$, $33$, $34$, $35$, $36$ and $37$.

We note that the converse of 
Corollary~\ref{maincorollary:even_ZariskiDense} is not true.
Indeed, for example, in the case $G=SL(n,\R)$ $(n\geq 3)$, 
there exist a non-even nilpotent orbit in $\mathfrak{g}$. 
However, by the result of Kim--Pansu (\cite[Theorem 1]{KIMPAN15}) above mentioned, for the corresponding non-even Lie group homomorphism $\rho\colon SL(2,\R)\rightarrow SL(n,\R)$, the representation $\rho|_{\Gamma{g}}$ can always be deformed into a Zariski-dense representation if the genus $g$ is sufficiently large.
Thus Theorem~\ref{maintheorem:deformation-zariski-closure} 
does not provide the optimal 
result for Problem~\ref{prob:Zariski-dense}. 

In the setting of Problem~\ref{prob:Zariski-dense}, 
by applying \cite[Th\'{e}or\`{e}me 2]{Guichard04}, we see that any small deformation of $\rho|_{\Gamma_{g}}$
is faithful and discrete if $\rho$ is non-trivial.
Hence we also remark that by combining Corollary~\ref{maincorollary:even_ZariskiDense} with the fact that any real non-compact semisimple Lie algebra admits some even nilpotent orbits, 
we obtain the following:
\begin{corollary}\label{corollary:surface_Zariskidens}
    A \reductive \ with a non-compact simple factor 
    contains a surface subgroup which is discrete and Zariski-dense.
\end{corollary}

\subsection{Non-standard deformation of a discontinuous group for a non-Riemannian homogeneous space}
\label{intro:non-standard}

In this subsection, we focus on 
the subproblem below of Problem \ref{problem:Gamma_nonstandard_Zariskidense} (\ref{item:problem:Gamma_nonstandard}):

\begin{problem}
\label{problem:deform-standard}
For a given standard discontinuous group $\Gamma$ for $G/H$ with non-compact $H$,
determine whether or not there exists a non-standard discontinuous group for $G/H$
sufficiently close to $\Gamma$.
\end{problem}

In contrast to the case where $H$ is compact, 
it is non-trivial whether a discrete subgroup of $G$ sufficiently close to $\Gamma$ act  
properly discontinuously on $G/H$ (\emph{stability for proper discontinuity}). 
Kobayashi \cite{Kobayashi93,Kobayashi98} and Kobayashi--Nasrin \cite{KobayashiNasrin06} 
initiated the study of the deformation of discontinuous groups for $G/H$ with non-compact $H$, 
and investigated local rigidity of discontinuous groups and stability for proper discontinuity.
Based on his properness criterion (Fact~\ref{fact:propernesscriterion}),
Kobayashi quantified ``the degree of proper discontinuity" of the $\Gamma$-action
and proved in \cite{Kobayashi98} stability for proper discontinuity in a certain broad setting.
In this paper, we apply a certain extension (Fact~\ref{fact:stability-for-properness}) of Kobayashi's theorem by Kassel~\cite{Kassel12}.

Related to Problem~\ref{problem:deform-standard}, 
non-trivial deformations of cocompact discontinuous groups $\Gamma$ for some specific irreducible semisimple symmetric spaces $G/H$ have been investigated. 
In the case where $H$ is compact, it is known as the Selberg--Weil theorem that
if the dimension of $G/H$ is greater than two, then any cocompact discontinuous group for $G/H$ is locally rigid.  
In contrast, in the case where $H$ is non-compact, some symmetric spaces $G/H$ of arbitrarily higher dimension
admit cocompact discontinuous groups which are not locally rigid. 
For example, let $G/H$ be the group manifold $(G' \times G')/\diag G'$ where $G'$ is locally isomorphic to $SO(n,1)$ or $SU(n,1)$.
We also fix a cocompact discrete subgroup $\Gamma'$ of $G'$ with $H^1(\Gamma',\R) \neq 0$.
Then Kobayashi \cite{Kobayashi98} proved that 
there exist non-standard small deformations of 
the cocompact discontinuous group $\Gamma=\Gamma' \times \{1\}$ for $G/H=(G' \times G')/\diag G'$. 
Note that Kobayashi's results above  give generalizations of the similar results  by Goldman \cite{Goldmannonstandard} for $G'=SL(2,\R)$ and those by Ghys \cite{Ghys95} for $G'=SL(2,\C)$.
As another example of $G/H$, 
let us consider the symmetric space $G/H  = SO(2,2n)/U(1,n)$ on which $SO(1,2n) \subset SO(2,2n)$ acts properly and cocompactly.
Kassel \cite{Kassel12} proved by applying Johnson--Millson's bending construction that 
for a specific arithmetic cocompact subgroup $\Gamma$ of $SO(1,2n)$, 
there exist Zariski-dense and non-standard small deformations of the cocompact discontinuous group $\Gamma$ for $G/H$.

It would be still worthwhile to consider Problem~\ref{problem:deform-standard} even if the Clifford--Klein form $\Gamma\backslash G/H$ is non-compact.
In this paper, we restrict ourselves to the case where $\Gamma$ is isomorphic to a surface group of sufficiently large genus, 
in which case $\Gamma\backslash G/H$ is far from being compact.
Instead we give an affirmative answer to Problem~\ref{problem:deform-standard}
for a general homogeneous space $G/H$ of reductive type as follows:

\begin{theorem}[see also Theorem \ref{theorem:strong-non-standard}]\label{theorem:any_sl2_nonstandard}
Let $G/H$ be a homogeneous space of reductive type with non-compact $H$,
and $\rho\colon SL(2,\R) \rightarrow G$ a Lie group homomorphism
such that $SL(2,\R)$ acts properly on $G/H$ via $\rho$.
Fix any discrete surface subgroup $\Gamma_{g}$ 
of $SL(2,\R)$ with sufficiently large genus $g$.
Then for any neighborhood $W$ of 
$\rho|_{\Gamma_{g}}$ in 
$\mathrm{Hom}(\Gamma_{g},G)$,
there exists a discrete and faithful representation $\rho'$ in $W$
such that $\rho'(\Gamma_{g})$ is a non-standard discontinuous group for $G/H$.
\end{theorem}

Let us outline the proof of Theorem~\ref{theorem:any_sl2_nonstandard}.
It follows from Kassel's stability theorem (Fact~\ref{fact:stability-for-properness}) for proper discontinuity  that 
the small deformation $\rho'\colon \Gamma_{g}\rightarrow G$ obtained by Theorem~\ref{maintheorem:deformation-zariski-closure} gives a
discontinuous group for $G/H$.
By applying Kobayashi's properness criterion (Fact~\ref{fact:propernesscriterion}),
we check that the Zariski closure $G^{\rho}_{\even}$ of $\rho'(\Gamma_{g})$
does \emph{not} act properly on $G/H$.
To prove the theorem, we shall show that  
for any \emph{not necessarily algebraic} reductive subgroup $L$ of $G$ (more generally, any \emph{closed Lie subgroup} $L$ of $G$ with finitely many connected components), $L$ does not act properly on $G/H$ if $L$ contains $\rho'(\Gamma_{g})$.
The proof of the theorem is completed by bridging the gap between the two notions of linear algebraic groups and linear Lie groups using Chevalley's theorem (Fact~\ref{fact:chevalley}).

We shall give an application of Theorem~\ref{theorem:any_sl2_nonstandard} 
to spectral analysis of intrinsic differential operators on Clifford--Klein forms $\Gamma\backslash G/H$,
which was recently initiated by a series of papers of Kassel and Kobayashi \cite{KasselKobayashi16,KasselKobayashi2019Invariant,KasselKobayashi2019standard,KasselKobayashi2020Spectral,Kobayashi16intrinsic} (see also \cite{Kobayashi22conjecture}).

Let $\Gamma$ be a discontinuous group for a reductive symmetric space $G/H$.
Then the Clifford--Klein form $\Gamma \backslash G/H$ becomes a pseudo-Riemannian locally symmetric space. 
As in Kassel--Kobayashi \cite{KasselKobayashi16}, 
the differential operators on $\Gamma \backslash G/H$ induced by $G$-invariant differential operators on $G/H$ 
are thought of as intrinsic on the $(G,G/H)$-manifold $\Gamma \backslash G/H$.
Kassel and Kobayashi have investigated the discrete spectrum $\mathrm{Spec}_d(\Gamma \backslash G/H)$,
that is, the set of joint $L^{2}$-eigenvalues of such intrinsic differential operators on $\Gamma \backslash G/H$ (see \cite[Section 1.1]{KasselKobayashi16} for a more precise definition).
It should be noted that the Laplacian on the pseudo-Riemannian manifold $\Gamma \backslash G/H$ is 
an important example of intrinsic differential operators and no longer an elliptic differential operator if $H$ is non-compact.

As a new direction of spectral analysis on locally symmetric spaces $\Gamma\backslash G/H$, 
Kassel and Kobayashi studied the behavior of $\mathrm{Spec}_d(\Gamma \backslash G/H)$
under small deformations of $\Gamma$, and
discovered in \cite{KasselKobayashi16} infinite discrete spectra on $\Gamma\backslash G/H$
which are stable under any small deformations of $\Gamma$ (\emph{stability for infinite discrete spectra}).
As an immediate corollary to \cite[Theorem 1.7]{KasselKobayashi16}
and Theorem \ref{theorem:any_sl2_nonstandard}, 
we obtain the following:

 \begin{corollary}
    \label{theorem:Zariskidense_infinitespectral}
    Let $G/H$ be a reductive symmetric space and $K$ a maximal compact subgroup of $G$
        such that $H$ is stable by the Cartan involution corresponding to $K$.
    Suppose that the following conditions hold:
    \begin{itemize}
        \item $G/H$ admits a discrete series representation (equivalently, the rank condition $\rank(G/H) = \rank(K /(K \cap H))$ holds),
        \item $G/H$ admits a proper $SL(2,\R)$-action via a Lie group homomorphism $\rho\colon SL(2,\R)\rightarrow G$.
    \end{itemize}
    Let $\Gamma_{g}$ be a discrete surface subgroup of $SL(2,\R)$ with a sufficiently large genus $g \geq 2$.
    Then there exists a neighborhood $W$ of $\rho|_{\Gamma_{g}}$ in $\Hom(\Gamma_{g},G)$
    satisfying the following:
    \begin{itemize}
    \item (Stability for proper discontinuity)
    For any $\rho'\in W$, the representation $\rho'$ is discrete and faithful, 
    and $\rho'(\Gamma_{g})$ is a discontinuous group for $G/H$.
    \item (Existence of non-standard small deformations)
    $W$ contains a representation $\rho'$ such that $\rho'(\Gamma_{g})$ is non-standard as a discontinuous group for $G/H$.
    \item (Stability for infinite discrete spectra)
    \[
    \# \bigcap_{\rho'\in W} \mathrm{Spec}_d(\rho'(\Gamma_{g}) \backslash G/H)  = \infty.
    \]
    \end{itemize}
\end{corollary}

For example, 
Corollary \ref{theorem:Zariskidense_infinitespectral} can be applied for the symmetric space $G/H = SU(p,q)/U(p,q-1)$ for $p\geq q$
(see also Section~\ref{section:examples}).

\subsection{Zariski-dense discontinuous surface group for a reductive symmetric space}
\label{section:intro-zariski-dense}

Let $G/H$ be a homogeneous space of reductive type
and $\Gamma_{0}$ a non-abelian free group or a surface group.
In this subsection,  
we focus on the existence problem (Problem~\ref{problem:Gamma_nonstandard_Zariskidense}~\eqref{item:problem:Gamma_Zariskidense})
of realizations of $\Gamma_{0}$ as Zariski-dense discrete subgroups of $G$ acting properly discontinuously on $G/H$.
Note that 
non-abelian free groups
and surface groups are both non virtually abelian,
that is, they have no finite-index abelian subgroups.

First, we recall the following fact 
which gives a Lie algebraic level criterion for $G/H$ admitting non virtually abelian discontinuous groups.  

\begin{fact}[Benoist {\cite[Theorem 1.1]{Benoist96}}]\label{fact:free_group_Zariski_dense}
In the setting above, $G$ contains a non virtually abelian discrete subgroup $\Gamma$ 
acting properly discontinuously on the homogeneous space $G/H$ of reductive type
if and only if 
\begin{equation}
\mathfrak{b}_+ \not \subset W \mathfrak{a}_\mathfrak{h}. \label{eq:Benoist}
\end{equation}
(see Section \ref{section:KobayashiBenoist} for the notations of $\mathfrak{b}_+$, $W$ and $\mathfrak{a}_\mathfrak{h}$).
In this case, we can choose $\Gamma$ to be free and Zariski-dense in $G$.
\end{fact}

It should be emphasized that 
the fact above gives an answer to Problem \ref{problem:Gamma_nonstandard_Zariskidense} \eqref{item:problem:Gamma_Zariskidense} for free groups $\Gamma_0$ as the condition \eqref{eq:Benoist}.

In this paper, we focus on 
Problem \ref{problem:Gamma_nonstandard_Zariskidense} \eqref{item:problem:Gamma_Zariskidense} for the case where $\Gamma_0$ is a surface group.
One of our main results, which gives an analogue of Fact \ref{fact:free_group_Zariski_dense} for surface groups, is the following:

\begin{theorem}[see also Theorem \ref{theorem:genZariskidenseSurface_on_symm}]\label{theorem:ZariskidenseSurface_on_symm}
Suppose that $G/H$ is a symmetric space 
and that $G$ contains a non virtually abelian discrete subgroup $\Gamma$
acting properly discontinuously on $G/H$.
Then we can choose $\Gamma$ to be isomorphic to a surface group and Zariski-dense in $G$.
\end{theorem}

Let us give an explanation of a proof of Theorem \ref{theorem:ZariskidenseSurface_on_symm} as below:
assume that $G$ contains a non virtually abelian discrete subgroup $\Gamma$
acting properly discontinuously on $G/H$
(or equivalently, $(G,H)$ satisfies the condition \eqref{eq:Benoist} above).
As discussed in \cite{Okuda13,Okuda17}, if $G/H$ is symmetric, then
there exists an \emph{even} Lie group homomorphism $\rho\colon SL(2,\R) \rightarrow G$ such that $SL(2,\R)$ acts properly on $G/H$ via $\rho$. 
In our proof of Theorem \ref{theorem:ZariskidenseSurface_on_symm}, 
by combining Corollary~\ref{maincorollary:even_ZariskiDense} with 
Kassel's stability theorem (Fact~\ref{fact:stability-for-properness}),
we obtain a discontinuous surface group for $G/H$ which is Zariski-dense in $G$.

\subsection{Organization of the paper}
In Section~\ref{section:preliminary}, 
we recall preliminary results for proper actions on homogeneous spaces of reductive type. 
In Section~\ref{section:zariski-closure-of-small-deformation}, we prove Theorem~\ref{theorem:deformation-zariski-closure}
generalizing Theorem~\ref{maintheorem:deformation-zariski-closure}. 
In Section~\ref{section:non-standard-deformation}, 
we prove Theorem~\ref{theorem:strong-non-standard} generalizing Theorem~\ref{theorem:any_sl2_nonstandard}. 
In Section~\ref{section:Zariski-dense}, 
we give a proof of Theorem~\ref{theorem:genZariskidenseSurface_on_symm}
generalizing Theorem~\ref{theorem:ZariskidenseSurface_on_symm}. 
In Section~\ref{section:examples}, 
as an example of Corollary~\ref{theorem:Zariskidense_infinitespectral} and Theorem~\ref{theorem:ZariskidenseSurface_on_symm}, 
we consider the classical symmetric space $G/H = SU(p,q)/U(p,q-1)$ for $p\geq q$ and 
give discontinuous groups for $G/H$ which are isomorphic to surface groups of high genus and Zariski-dense in $G$. 
In Appendix, 
we give some remarks on 
Fact~\ref{fact:free_group_Zariski_dense} in the case where $G$ is not semisimple but reductive. 

\section{Preliminaries}
\label{section:preliminary}

\subsection{Terminology}
\label{subsection:terminology}

We mean by a \emph{\reductive} 
a (Euclidean) open subgroup $G$ of
the real points $\mathbf{G}(\R)$
of a connected reductive algebraic group $\mathbf{G}$ defined over $\R$. 
When we state ``$G$ is a \reductive", we also implicitly consider $\mathbf{G}$.
By the Zariski topology of $G$, we mean the relative topology of the Zariski topology on the complex points $\mathbf{G}(\C)$. 
For a subset $S$ of $G$, we denote by $\overline{S}^{\Zariski}$ the closure (\emph{Zariski closure}) of $S$ with respect to the Zariski topology of $G$.
The subset $S$ is called \emph{Zariski-dense} if $\overline{S}^{\Zariski}=G$.

A subgroup $H$ of $G$ is \emph{reductive}
if $H$ is a (not necessarily algebraic) closed subgroup of finitely many connected components
with respect to the Euclidean topology which is stable 
under a Cartan involution of $G$.
Then the homogeneous space $G/H$ is called 
\emph{of reductive type}.

A subgroup $H$ of $G$ is called \emph{symmetric} if $H$ is a (Euclidean) open subgroup of the fixed subgroup of an involution of $G$. 
The pair $(G,H)$ is called a symmetric pair and $G/H$ is called a \emph{reductive symmetric space}.

\subsection{Preliminaries for the properness criterion}\label{section:KobayashiBenoist}

Throughout this subsection, let us fix $G$ as a real reductive group in the sense of Section \ref{subsection:terminology}. 
In this section, we recall some terminologies and results given by T.~Kobayashi \cite{Kobayashi96}, Y.~Benoist \cite{Benoist96} and F.~Kassel \cite{Kassel12} in a form that we shall need.

First, we set up a binary relation ``$\pitchfork$'' defined on the power set of $G$ as follows: 
\begin{definition}[Kobayashi \cite{Kobayashi96}]\label{definition:tilde_fork}
Let $L$ and $H$ be both subsets in $G$.
We write \[
    L \pitchfork H \text{ in } G \]
    if for any compact subset $C$, 
    the subset $L \cap C H C$ is relatively compact in $G$.
\end{definition}

One can easily check that 
for closed subgroups $L, H$ of $G$,
the $L$-action on $G/H$ is proper if and only if $L \pitchfork H$ in $G$.

We write $\mathfrak{g}$ for the Lie algebra of $G$,
and fix a Cartan involution $\theta$ on $G$. 
The Cartan decomposition of $\mathfrak{g}$ with respect to $\theta$ is written as $\mathfrak{g} = \mathfrak{k} + \mathfrak{p}$.
Let $\mathfrak{a}$ be a maximal abelian subspace of $\mathfrak{p}$. 
In this paper, we call such $\mathfrak{a}$ a maximal split abelian subspace of $\mathfrak{g}$.
The restricted root system of $(\mathfrak{g},\mathfrak{a})$ 
and its Weyl group acting on $\mathfrak{a}$
are denoted by $\Sigma = \Sigma(\mathfrak{g},\mathfrak{a})$ and $W = W(\mathfrak{g},\mathfrak{a})$,
 respectively.
It should be noted that the $W$-action is trivial on $\mathfrak{z}(\mathfrak{g}) \cap \mathfrak{a}$,
where $\mathfrak{z}(\mathfrak{g})$ denotes the center of $\mathfrak{g}$.

Let us also fix a positive system $\Sigma_+$ of $\Sigma$, and define 
\[
\mathfrak{a}_+ := \{ A \in \mathfrak{a} \mid \langle A, \alpha \rangle \geq 0 \text{ for any } \alpha \in \Sigma_+ \}.
\]
for the closed Weyl chamber of $\mathfrak{a}$ corresponding to $\Sigma_+$.
The Cartan projection is written as 
\[
\mu\colon G \rightarrow \mathfrak{a}_+, ~ g \mapsto \mu(g), 
\]
where
$\mu(g)$ denotes the unique element of $\mathfrak{a}_+$ such that 
\[
g \in K (\exp \mu(g)) K.
\]

For closed subgroups $L$ and $H$ of $G$,  
the following fact gives a criterion for the properness 
of the $L$-action on the homogeneous space $G/H$ 
in terms of the images of $L$ and $H$ under the Cartan projection:

\begin{fact}[Properness criterion]\label{fact:propernesscriterion}
Let $L$ and $H$ be both subsets in $G$.
Then the following two conditions are equivalent:
    \begin{enumerate}
        \item \label{item:properness:proper} $L \pitchfork H$ in $G$.
        \item \label{item:properness:mu} $\mu(L) \pitchfork \mu(H)$ in the finite-dimensional vector space $\mathfrak{a}$. 
    \end{enumerate}
\end{fact}

\begin{remark}
For the case where $L$ and $H$ are both reductive subgroups of $G$, 
the properness criterion for the $L$-action on $G/H$ was given by Kobayashi \cite[Theorem 4.1]{Kobayashi89} in terms of $KAK$-decompositions of $L$ and $H$.
As its generalization, 
Fact \ref{fact:propernesscriterion} was proved by 
Kobayashi \cite{Kobayashi96} 
and 
Benoist \cite{Benoist96}.
\end{remark}

We shall fix a reductive subgroup $H$ of $G$
and consider the homogeneous space $G/H$ of reductive type in the sense of Section \ref{subsection:terminology}. 
One can assume that $H$ is $\theta$-stable and $\mathfrak{a}_\mathfrak{h} := \mathfrak{a} \cap \mathfrak{p}_\mathfrak{h}$ is a maximal abelian subspace of $\mathfrak{p}_\mathfrak{h}$, where $\mathfrak{h} = \mathfrak{k}_\mathfrak{h} + \mathfrak{p}_\mathfrak{h}$ denotes the Lie algebra of $H$ with the Cartan decomposition corresponding to $\theta$. 
Then by considering the $KAK$-decomposition of $H$, we have 
\[
\mu(H) = W \mathfrak{a}_\mathfrak{h} \cap \mathfrak{a}_+.
\]

Then as a corollary to Fact \ref{fact:propernesscriterion},
we have the following claim which will applied in Section \ref{section:Zariski-dense}:
\begin{corollary}\label{cor:sl_2-proper}
Let $\rho\colon SL(2,\R) \rightarrow G$ be a Lie group homomorphism with $\rho(A_0) \in \mathfrak{a}_+$, where we put 
\[
A_0 := \begin{pmatrix} 1 & 0 \\ 0 &  -1\end{pmatrix} \in \mathfrak{sl}(2,\R).
\]
Then the $SL(2,\R)$-action on $G/H$ via $\rho$ is proper if and only if $\rho(A_0) \not \in W \mathfrak{a}_\mathfrak{h}$.
\end{corollary}

Let us write $w_0$ for the longest element of $W$ with respect to $\Sigma_+$, and define the opposition involution 
\[
\iota\colon \mathfrak{a} \rightarrow \mathfrak{a}, ~ A \mapsto -w_0(A).
\]
Then the closed Weyl chamber $\mathfrak{a}_+$ is stable by $\iota$.
Let us put 
\[
\mathfrak{b} := \{ A \in \mathfrak{a} \mid \iota(A) = A \}.
\]
and $\mathfrak{b}_+ := \mathfrak{b} \cap \mathfrak{a}_+$. 
Note that the opposition involution $\iota$ 
acts as $-1$ on $\mathfrak{z}(\mathfrak{g}) \cap \mathfrak{a}$.
In particular, $\mathfrak{b}$ is contained in the semisimple part $\mathfrak{g}_{\mathrm{ss}} := [\mathfrak{g},\mathfrak{g}]$ of $\mathfrak{g}$.

The following fact gives a criterion for existence of non-virtually abelian discontinuous groups for $G/H$ in terms of $\mathfrak{b}_+$ and $\mathfrak{a}_\mathfrak{h}$:

\begin{fact}[{\cite[Th\'{e}or\`{e}me 1.1]{Benoist96}}]
\label{fact:Benoist}
The following three conditions on $(G,H)$ are equivalent:
\begin{enumerate}
    \item $G/H$ admits a discontinuous group which is not virtually abelian.
    \item $G/H$ admits a discontinuous group which is a non-abelian free group and Zariski-dense in $G$.
    \item $\mathfrak{b}_+ \not \subset W \mathfrak{a}_\mathfrak{h}$.
\end{enumerate}
\end{fact}

\begin{remark}
Benoist \cite{Benoist96} gives a proof of Fact \ref{fact:Benoist} only for the case where $G$ is semisimple.
In Appendix, we will give some remarks for Fact \ref{fact:Benoist} in our setting. 
\end{remark}

For stability for proper discontinuity,
the following fact generalizing Kobayashi \cite[Theorem 2.4]{Kobayashi98}, which will plays a key role 
in our proofs of Theorems \ref{theorem:any_sl2_nonstandard} and \ref{theorem:ZariskidenseSurface_on_symm}, 
is known:

\begin{fact}[Kassel {\cite[Theorem 1.3]{Kassel12}}]
\label{fact:stability-for-properness}
Let $L$ be a connected linear simple Lie group of real rank $1$, and $\rho\colon L\rightarrow G$ a Lie group homomorphism 
such that the $L$-action on $G/H$ via $\rho$ is proper.
Assume that a torsion-free discrete subgroup $\Gamma$ 
of $L$ is cocompact in $L$.
Then there exists a neighborhood $W$ of 
$\rho|_{\Gamma}$ in $\operatorname{Hom}(\Gamma,G)$
such that 
for any $\rho'\in W$, the representation $\rho'$ is discrete and faithful, 
and that the $\Gamma$-action on $G/H$ via $\rho'$
is properly discontinuous.
\end{fact}

\section{Deformations of surface subgroups of real reductive groups}
\label{section:zariski-closure-of-small-deformation}

Let $\Sigma_{g}$ be a closed oriented surface of 
genus $g\geq 2$. In this section, we study a representation of the surface group $\pi_{1}(\Sigma_{g})$  into a \reductive~$G$ 
whose Zariski closure is locally isomorphic to $SL(2,\R)$. 
We give its small deformations in $G$ by a bending construction (Lemma \ref{lemma:deform-of-surface-group} (\ref{lemma:deform-of-surface-group-def})), 
and determine their Zariski closures (Theorem \ref{theorem:deformation-zariski-closure}).
Theorem \ref{theorem:deformation-zariski-closure} will play an important role in our proofs of Theorems \ref{theorem:any_sl2_nonstandard} and \ref{theorem:ZariskidenseSurface_on_symm}.

Throughout this section, we use the symbol $A_0$ for the matrix 
\[
A_0 := \begin{pmatrix} 1 & 0 \\ 0 &  -1\end{pmatrix} \in \mathfrak{sl}(2,\R).
\]

\subsection{Zariski closures of surface subgroups in a \reductive.}

In this subsection, we explain
a key result (Theorem \ref{theorem:deformation-zariski-closure}). 
Let us fix a real reductive group $G$ with Lie algebra $\mathfrak{g}$ in the sense of Section \ref{subsection:terminology}.
We first define a Lie subalgebra $\mathfrak{g}^{\rho}_{\even}$ of $\mathfrak{g}$ associated to a Lie group homomorphism $\rho\colon SL(2,\R)\rightarrow G$.

\begin{definition}\label{definition:geven}
For $k\in\Z$ and $A\in\mathfrak{g}$, we put 
\[
\mathfrak{g}(k;A) := \{ X \in \mathfrak{g} \mid [A,X] = k X \}.
\]
For each Lie algebra homomorphism $\rho\colon \mathfrak{sl}(2,\R) \rightarrow \mathfrak{g}$, let us define 
\[
\mathfrak{g}^\rho_{\even} := \bigoplus_{k \in \Z} \mathfrak{g}(2k;\rho(A_0)).
\]
\end{definition}

For $k\in\mathbb{N}$ let $V_{k}$ 
be the real form of a $k$-dimensional complex irreducible 
representation of $SL(2,\mathbb{R})$.
The following is obvious:
\begin{lemma}
\label{lemma:rho-even}
Let $\rho\colon SL(2,\R)
\rightarrow G$ be 
a Lie group homomorphism, and 
we regard the Lie algebra $\mathfrak{g}$ as 
an $SL(2,\R)$-module by the adjoint action
$\mathrm{Ad}\circ\rho$.
Then $\mathfrak{g}^{\rho}_{\even}$
is the sum of $SL(2,\R)$-submodules 
of $\mathfrak{g}$ isomorphic to $V_{2i+1}$
for some $i\in\N$.
\end{lemma}

Here, we prove the statement announced in Definition~\ref{intro_def:rho-G-g} and that $\mathfrak{g}^\rho_{\even}$ coincides with the Lie algebra of $G_{\even}^{\rho}$. 
\begin{lemma}
\label{lemma:rho-G-g}
In the setting of Definition \ref{intro_def:rho-G-g} the following hold:
\begin{enumerate}
    \item 
    $\sigma\in \mathbf{G}(\R)$.
    \item
    The algebraic subgroup $\mathbf{G}_{\even}^{\rho}$ of $\mathbf{G}$ is defined over $\R$.
    \item \label{item:lemma:rho-G-g:LiealgGrhoeven}
    $\mathfrak{g}^\rho_{\even}$ coincides with
    the Lie subalgebra of $\mathfrak{g}$
    corresponding to the subgroup $G_{\even}^{\rho}$.
\end{enumerate}
\end{lemma}

\begin{proof}

\textbf{(i).} We fix a faithful rational
representation $\phi$ of 
$\mathbf{G}$ defined over $\R$, and
denote the complex conjugate 
of $g\in \mathbf{G}(\C)$ by $\bar{g}$.
Since $\phi(\rho(A_{0}))$ is semisimple and 
all its eigenvalues are integers,
the image of $\sigma\overline{\sigma}^{-1}=
\mathrm{exp}(2\pi\sqrt{-1}\rho(A_{0}))$ 
under $\phi$ is the identity matrix.
Hence $\overline{\sigma}=\sigma$, 
and thus we obtain $\sigma \in \mathbf{G}(\R)$. 

\noindent\textbf{(ii).}
It follows from (i) that the centralizer of $\sigma$ is defined over $\R$.
Hence we see from \cite[Proposition 1.2 (b)]{Borel-alg-grp} that its identity component
$\mathbf{G}_{\even}^{\rho}$ is also defined over $\R$.

\noindent\textbf{(iii).}
For $k\in \N$ and $X\in \mathfrak{g}(k;\rho(A_0))$, we have
\[
\exp(\pi\sqrt{-1}\mathrm{ad}(\rho(A_{0})))(X)=(-1)^{k}X.
\]
Hence we obtain
\begin{align*}
\mathfrak{g}^{\rho}_{\even}&=\{X\in \mathfrak{g}\mid
\exp(\pi\sqrt{-1}\mathrm{ad}(\rho(A_{0}))(X)=X\} \\
&= \{X\in \mathfrak{g}\mid
\mathrm{Ad}(\sigma)(X)=X\}, 
\end{align*}
which is the Lie subalgebra corresponding to the subgroup $G^{\rho}_{\even}$.
\end{proof}

The following lemma will be applied in Section \ref{section:non-standard-deformation}:
\begin{lemma}
\label{lemma:grhoev_fulrank}
The equality 
\[
\rank_{\R} G^{\rho}_{\even} = \rank_\R G
\]
holds, that is, 
there exists a maximal split abelian subspace $\mathfrak{a}$ of $\mathfrak{g}$ contained in $\mathfrak{g}^{\rho}_{\even}$.
\end{lemma}

\begin{proof}
One can easily check that $\rho(A_0)$ is hyperbolic in $\mathfrak{g}$,
and hence there exists a maximal split abelian subspace $\mathfrak{a}$ of $\mathfrak{g}$ with $\rho(A_0) \in \mathfrak{a}$.
Then we have 
\[
\mathfrak{a} \subset \mathfrak{g}(0;\rho(A_0)) \subset \mathfrak{g}^\rho_{\even}.
\]
This proves our claim.
\end{proof}

\begin{definition}
    \label{definition:fcc-closure}
    Let $G$ be a Lie group with finitely many connected components
and $\mathfrak{g}$ the Lie algebra of $G$.
For each subgroup $\Gamma$ of $G$, 
we denote by $\overline{\Gamma}^{\mathrm{fcc}}$ the intersection of closed Lie subgroups of $G$ with finitely many connected components which contain
$\Gamma$.
\end{definition}

\begin{remark}
    The Lie group
    $\overline{\Gamma}^{\text{fcc}}$ 
    may have infinitely many connected components.
\end{remark}

For a finite-dimensional real representation $V$ of 
$SL(2,\mathbb{R})$,
we denote by $[V:V_{k}]$ the multiplicity of the $k$-dimensional real irreducible representation $V_{k}$ in $V$.
The goal of this section is to prove the following theorem generalizing Theorem \ref{maintheorem:deformation-zariski-closure}: 
\begin{theorem}
\label{theorem:deformation-zariski-closure}
Let $G$ be a \reductive,
$\rho\colon SL(2,\mathbb{R})\rightarrow G$ 
a Lie group homomorphism, 
and $\Gamma_{g}$ a discrete subgroup of $SL(2,\mathbb{R})$ isomorphic to $\pi_{1}(\Sigma_{g})$ with genus $g\geq 2$.
Let $G'$ be a Zariski-connected real algebraic subgroup of $G$ satisfying 
\begin{align}
\label{assumption:SL-G'-Grho}
\rho(SL(2,\R)) \subset G' \subset G^{\rho}_{\even},
\end{align}
and we regard its Lie algebra $\mathfrak{g}'$ as an $SL(2,\mathbb{R})$-module by the adjoint 
action $\Ad\circ\rho$. 
Assume 
\begin{align}
    \label{eq:genus-condition}
    g\geq \sum_{i\in\mathbb{N}} [\mathfrak{g}':V_{2i+1}].
\end{align}
Then the following claims hold:
\begin{enumerate}
    \item For any neighborhood $W$ of $\rho|_{\Gamma_{g}}$ 
    in $\mathrm{Hom}(\Gamma_{g},G)$, there exists $\rho'\in W$ such that 
    the Zariski closure $\overline{\rho'(\Gamma_{g})}^{\Zariski}$ coincides with $G'$.
    \item Furthermore assume that $\mathfrak{g}'$ contains the centralizer $\mathfrak{z}_{\mathfrak{g}}(\rho(\mathfrak{sl}(2,\R)))$. Then for any neighborhood $W$ of $\rho|_{\Gamma_{g}}$ in $\mathrm{Hom}(\Gamma_{g},G)$, there exists $\rho'\in W$ such that $\rho'$ satisfies the condition in (i), and that $\overline{\rho'(\Gamma_{g})}^{\mathrm{fcc}}$ contains the Euclidean identity component $G'_{0}$ of $G'$.
\end{enumerate}
\end{theorem}

As an immediate consequence of the theorem (consider the case where $G'=G^{\rho}_{\even}$), 
we get the following two known examples: 

\begin{example}
\label{example:deform-trivial}
If $g\geq \dim G$, there exists a Zariski-dense representation of $\pi_{1}(\Sigma_{g})$ to $G$
sufficiently close to the trivial representation.
\end{example}

\begin{example}
\label{example:deform-hitchin}
Let $\rho\colon SL(2,\R)\rightarrow SL(n,\R)$ be an $n$-dimensional irreducible representation of $SL(2,\R)$
and $\Gamma_{g}$ a discrete subgroup of $SL(2,\R)$ isomorphic to $\pi_{1}(\Sigma_{g})$.
If $g\geq n-1$, then
there exists a discrete faithful representation $\rho'\in\Hom(\Gamma_{g},SL(n,\R))$
sufficiently close to $\rho|_{\Gamma_{g}}$ such that $\rho'(\Gamma_{g})$ is Zariski-dense in $SL(n,\R)$.
\end{example}

The following obvious lemma will be applied in Section \ref{section:examples}
to compute the right hand side of the genus assumption (\ref{eq:genus-condition}):
\begin{lemma}\label{lem:lower-bound-of-genus}
    In the setting of Theorem \ref{theorem:deformation-zariski-closure}, one has
    \[
    \sum_{i\in\mathbb{N}} [\mathfrak{g}:V_{2i+1}] = \dim_{\R} \{X\in \mathfrak{g}\mid [\rho(A_{0}),X]=0 \}.
    \]
\end{lemma}

\subsection{Lemmas for the proof of 
Theorem \ref{theorem:deformation-zariski-closure}}
In this subsection, we prove three lemmas (Lemmas~\ref{lemma:deform-of-surface-group}, \ref{lemma:Zariski-dense-basis-of-AG-construction} and \ref{lemma:Zariski-dense-basis-of-AG}) for 
the proof of Theorem \ref{theorem:deformation-zariski-closure}.

We often apply the following fact:
\begin{fact}[Chevalley {\cite[Chapitre 2, Th\'{e}or\`{e}me 13]{Chevalley51}}]
\label{fact:chevalley}
    Let $G \subset GL(n,\R)$ be a closed subgroup with finitely many connected components.
    We denote by $\mathfrak{g}$ and $\overline{\mathfrak{g}}^{\Zariski}$ 
    the real Lie algebra of $G$ and that of its Zariski closure $\overline{G}^{\Zariski}$ in $GL(n,\R)$, respectively. Then we have 
    \[
    [\overline{\mathfrak{g}}^{\Zariski},\overline{\mathfrak{g}}^{\Zariski}]= [\mathfrak{g},\mathfrak{g}].
    \]
\end{fact}

For each $g \geq 2$, 
we shall fix a generating set $\{ a_1,b_1,\dots,a_g,b_g \}$ 
of the surface group $\Gamma_g$ such as 
\begin{align}
\label{eq:generators-and-relations-surface-group}
\Gamma_{g}=
\langle a_{1},b_{1},\ldots,a_{g},b_{g}\mid 
[a_{1},b_{1}]\cdots[a_{g},b_{g}]=1 \rangle.
\end{align}
\begin{lemma}
\label{lemma:deform-of-surface-group}
Let us fix a representation $\rho_0 \colon \Gamma_{g}\rightarrow G$ of $\Gamma_{g}$
to a Lie group $G$.
Assume that $X_{1},\ldots,X_{g}\in\mathfrak{g}$
satisfies 
\begin{align}
\label{condition:deform-of-surface-group}
\mathrm{Ad}(\rho(a_{k}))(X_{k})=X_{k}
\ \ (1\leq k \leq g).
\end{align}
Then the following two assertions hold:
\begin{enumerate}
    \item
    \label{lemma:deform-of-surface-group-def}
    For each $t\in\R$, we put
\begin{align}
\label{def:rho-t}
\rho_{t}(a_{k}):= \rho_0(a_{k}),\ 
\rho_{t}(b_{k}):= \rho_0(b_{k})e^{tX_{k}}
\ \ (1\leq k \leq g).
\end{align}
Then $\rho_{t}$ defines a group homomorphism from
$\Gamma_{g}$ to $G$.
    \item 
    \label{lemma:deform-of-surface-group-property}
    We consider the case where $\Gamma_{g}$ is a discrete subgroup of $SL(2,\R)$.
    Assume that $G$ is a linear real algebraic group and that 
    there exists a Lie group homomorphism $\rho\colon SL(2,\R) \rightarrow G$ satisfying $\rho|_{\Gamma_{g}}=\rho_{0}$.
    Then for any $t \in \R$, the Lie group $\overline{\rho_{t}(\Gamma_{g})}^{\mathrm{fcc}}$ contains 
    $\rho(SL(2,\R))$ and $e^{tX_i}$ for $i = 1,\dots,g$. 
\end{enumerate}
\end{lemma}

\begin{proof}
Let us first prove (\ref{lemma:deform-of-surface-group-def}).
By the condition (\ref{condition:deform-of-surface-group})
and the formula (\ref{def:rho-t})
we see
\[
[\rho_{t}(a_{k}),\rho_{t}(b_{k})]=
[\rho(a_{k}),\rho(b_{k})]=\rho([a_{k},b_{k}])
\]
for any $1\leq k\leq g$. Thus we have
\[
[\rho_{t}(a_{1}),\rho_{t}(b_{1})]\cdots
[\rho_{t}(a_{g}),\rho_{t}(b_{g})]=
\rho([a_{1},b_{1}]\cdots[a_{g},b_{g}])
=1.
\]
Hence $\rho_{t}$ defines a group homomorphism from $\Gamma_{g}$
to $G$ thanks to
the presentation \eqref{eq:generators-and-relations-surface-group}.

Next we prove (\ref{lemma:deform-of-surface-group-property}).
Let $L_{t}$ be a Lie group with finitely many connected components which contains $\rho_{t}(\Gamma_{g})$.
It suffices to show that $e^{tX_{i}} \in L_{t}$
and $\rho(SL(2,\mathbb{R}))\subset L_{t}$.
The assertion $e^{tX_{i}} \in L_{t}$ follows from the assertion $\rho(SL(2,\mathbb{R}))\subset L_{t}$.
Indeed, if $\rho(SL(2,\mathbb{R}))\subset L_{t}$, then $\rho_{0}(b_{i})=\rho(b_{i})\in L_{t}$ for each $i=1,\ldots,g$. 
Hence we have $e^{tX_{i}}=\rho_{0}(b_{i})^{-1}\rho_{t}(b_{i})\in L_{t}$.

Let us show $\rho(SL(2,\mathbb{R}))\subset L_{t}$.
First suppose that $L_{t}$ is a real algebraic subgroup of $G$.
Then it follows from the elementary lemma below (Lemma \ref{lemma:zariski-dense-surface-group-sl2}) that
the subgroup generated by $a_{1},a_{2}\in \Gamma_{g}$ 
is Zariski-dense in $SL(2,\mathbb{R})$.
Since $\rho(a_{i})(=\rho_{t}(a_{i}))$ belongs to $L_{t}$ for each $i=1,2$,
we have $\rho(SL(2,\mathbb{R}))\subset L_{t}$.

For general $L_{t}$ with Lie algebra $\mathfrak{l}_{t}$, it follows from the above that 
the Zariski closure $\overline{L_{t}}^{\Zariski}$ contains $\rho(SL(2,\R))$.
Hence its Lie algebra $\overline{\mathfrak{l}_{t}}^{\Zariski}$ also contains $\rho(\mathfrak{sl}(2,\R))$.
By Fact \ref{fact:chevalley}, we obtain
\[
\mathfrak{l}_{t}\supset[\mathfrak{l}_{t},\mathfrak{l}_{t}]=[\overline{\mathfrak{l}_{t}}^{\Zariski},\overline{\mathfrak{l}_{t}}^{\Zariski}]
\supset [\rho(\mathfrak{sl}(2,\R)),\rho(\mathfrak{sl}(2,\R))]=\rho(\mathfrak{sl}(2,\R)).
\]
Thus $L_{t}$ also contains $\rho(SL(2,\R))$. This proves the assertion (ii).
\end{proof}

In this proof, we used the following lemma (we omit its proof): 
\begin{lemma}
\label{lemma:zariski-dense-surface-group-sl2}
Any non-abelian subgroup of hyperbolic elements of $SL(2,\R)$ 
is Zariski-dense.
\end{lemma}

As in the sense of \cite[Chapter IX, Section 7]{HelgasonDiffLieSymm2001}, 
an invertible matrix $g \in GL(n,\R)$ is called elliptic (resp. hyperbolic) if $g$ is diagonalizable over $\C$ with eigenvalues of norm $1$ (resp. with positive real eigenvalues). 
For a linear algebraic group $\mathbf{G}$ defined over $\R$, an element $g$ in $\mathbf{G}(\R)$ is called elliptic, hyperbolic, or unipotent if for some faithful $\R$-rational representation $\rho$ of $\mathbf{G}$, the matrix $\rho(g)$ is so, respectively. 
It should be noted that if $g \in \mathbf{G}(\R)$ is elliptic, hyperbolic, or unipotent, 
then for any $\R$-morphism $\phi\colon \mathbf{G}\rightarrow \mathbf{G}'$ of linear $\R$-algebraic groups,
the element $\phi(g)\in \mathbf{G}'(\R)$ is so, respectively.
An element $X$ in the real Lie algebra $\mathfrak{g}$ of $\mathbf{G}$ is called elliptic or hyperbolic if $\exp(tX)\in \mathbf{G}(\R)$ is so for any $t \in \R$, respectively, and is called nilpotent if $\exp(tX)\in \mathbf{G}(\R)$ is unipotent for any $t \in \R$.

Now we fix a linear algebraic group $\mathbf{G}$
    defined over $\R$, and denote by $\mathfrak{g}$ the real Lie algebra of the Lie group $\mathbf{G}(\R)$.
For subsets of $\mathfrak{g}$, 
let us introduce the following notation: 

\begin{definition}\label{definition:propertystar}
    We say that a finite subset $\mathcal{B}$
    of $\mathfrak{g}$
    satisfies Property~$(*)$ 
    if the three conditions below hold:
    \begin{enumerate}
        \item Any element $X$ in $\mathcal{B}$ is elliptic, hyperbolic, or nilpotent.
        \item For any elliptic $X\in\mathcal{B}$ 
        we have $\exp(X)=1$.
        \item The real vector space $\mathfrak{g}$ 
        is spanned by $\mathcal{B}$.
    \end{enumerate}
\end{definition}

\begin{lemma}
\label{lemma:Zariski-dense-basis-of-AG-construction}
    The Lie algebra $\mathfrak{g}$ admits a real basis $\mathcal{B}$ with Property~$(*)$ in Definition~\ref{definition:propertystar}.
\end{lemma}
\begin{proof}
    Without loss of generality, we may assume that $\mathbf{G}$ is connected.
    Then by using the Levi decomposition of $\mathbf{G}$ (Mostow's theorem \cite{Levidecomposition}), 
    it suffices to prove the claim in the two cases where 
    $\mathbf{G}$ is unipotent and where $\mathbf{G}$ is reductive.

    If $\mathbf{G}$ is unipotent, any element $X$ in $\mathfrak{g}$ is nilpotent. 
    Thus any basis of $\mathfrak{g}$ satisfies the desired conditions.
    
    If $\mathbf{G}$ is reductive, then 
    we fix a Cartan involution $\theta$ of $\mathbf{G}(\R)$ by Mostow's theorem \cite[Theorem 7.3]{Mostow_self-adjoint}.
    Let $\mathfrak{g} = \mathfrak{k} + \mathfrak{p}$ be the corresponding Cartan decomposition of $\mathfrak{g}$.
    Namely, $\mathfrak{k}$ (resp.\ $\mathfrak{p}$) is the eigenspace of the differential of $\theta$
    with eigenvalue $+1$ (resp.\ $-1$). Then $K = \{g\in \mathbf{G}(\R)\mid \theta(g) = g\}$ is a maximal compact
    subgroup of $\mathbf{G}(\R)$. We consider its Euclidean identity component $K_{0}$.

    Let $T$ be a maximal torus of $K_{0}$ and we put $t=\dim T$. Take a $\Z$-basis $\{Y_{1},\ldots,Y_{t}\}$
    of $\mathrm{ker}(\mathrm{exp}\colon \mathfrak{t} \rightarrow T)$.
    Then thanks to the Cartan decomposition
    $\displaystyle K_{0}=\bigcup_{k\in K_{0}}kTk^{-1}$, 
    one can take a basis $\mathcal{B}_{\mathfrak{k}}$ of $\mathfrak{k}$
    contained in $\{\mathrm{Ad}(k)(Y_{i})
    \mid k \in K_{0},\ 1\leq i \leq t\}$.
    We also fix a basis $\mathcal{B}_{\mathfrak{p}}$ of $\mathfrak{p}$.
    Then $\mathcal{B}:=\mathcal{B}_{\mathfrak{k}}\cup\mathcal{B}_{\mathfrak{p}}$ satisfies the desired properties for reductive $\mathbf{G}$.
    This proves our claim.
\end{proof}

\begin{lemma}
\label{lemma:Zariski-dense-basis-of-AG}
    Let $\mathcal{B}$ be a subset of $\mathfrak{g}$ with Property~$(*)$ in Definition \ref{definition:propertystar}.
    Fix an arbitrary $t\in\R\smallsetminus\Q$ and  
    let $L$ be a closed Lie subgroup of $\mathbf{G}(\R)$ which 
    has finitely many connected components and contains $\exp(tX)$ for all $X\in \mathcal{B}$. 
    Then the Lie algebra of $L$ coincides with $\mathfrak{g}$
    in the following two cases:
    \begin{description}
        \item[Case 1] $\mathbf{G}$ is commutative. 
        \item[Case 2] $L$ is an algebraic subgroup of $\mathbf{G}(\R)$.
    \end{description}
\end{lemma}

\begin{proof}
    \noindent (Proof in Case 1)
    Since $L$ has finitely many connected components, 
    one can find a positive integer $n$ such that $\exp(nt\mathcal{B})$ is contained in the Euclidean identity component of $L$.
    Therefore, we may assume, by replacing $\mathcal{B}$ with $n\mathcal{B}$, that the Lie group $L$ is connected. Then $L$ is contained in the Euclidean identity component $G$ of $\mathbf{G}(\R)$.

    Let $K$ be the unique maximal compact Lie subgroup of the connected abelian Lie group $G$, which coincides with the set of elliptic elements of $G$. Then the Lie algebra of $K$ is spanned by the elliptic elements 
    in $\mathcal{B}$. For elliptic $X\in \mathcal{B}$
    it follows from $\exp(X)=1$ and $t\not\in \Q$ 
    that $\exp(\Z tX)(\subset L)$ is dense in $\exp(\R X)$ with respect to the Euclidean topology.
    Hence the closed subgroup $L$ contains $K$. Here we notice that the Lie group $G/K$ is isomorphic to a vector space and that the vector space $G/K$ is spanned by the images of 
    $\exp(tX)\in L$ for all $X\in \mathcal{B}$.
    Since the Lie group $L/K$ is closed in $G/K$ and connected, 
    we have $L/K = G/K$. Hence we obtain $L=G$, which proves our claim.
    \vspace{\baselineskip}

    \noindent (Proof in Case 2) 
    It suffices to show that $\exp(\Z tX)$ is Zariski-dense in $\exp(\R X)$ for all $X\in \mathcal{B}$. 
    In the case where $X$ is elliptic, the Euclidean density (and thus the Zariski density) follows as seen above. If $X$ is hyperbolic, 
    the Zariski density follows from the following Chevalley's lemma 
    (e.g., see \cite[Lemma 2.A.1.2]{Wallach-real-reductive-group}):
    for a symmetric matrix $X\in M(N,\R)$, if a polynomial
    of matrix entries vanishes on $\mathrm{exp}(\Z X)$, 
    then it also does so on $\mathrm{exp}(\R X)$.
    Finally, if $X$ is nilpotent, the Zariski density is implied by 
    the fact that the map $t\mapsto \exp(tX)$ from $\R$ to $G$ is a polynomial map.
    \end{proof}

%

\subsection{Proof of Theorem \ref{theorem:deformation-zariski-closure}}

In this subsection, we give a proof of Theorem \ref{theorem:deformation-zariski-closure}.
We follow the setting and notation of Theorem \ref{theorem:deformation-zariski-closure}.

We need a few notations. Put 
\[
\Lambda :=\{(i,j)\in \N\times \N \mid [\mathfrak{g}':V_{2i+1}]>0,\ 1\leq j\leq 
[\mathfrak{g}':V_{2i+1}]\}.
\]
Applying Lemma \ref{lemma:Zariski-dense-basis-of-AG-construction} to the centralizer $Z:=Z_{G'}(\rho(SL(2,\R)))$, 
we take a basis $\{X_{0,1},\ldots, X_{0,[\mathfrak{g}':V_{1}]}\}$ of its real Lie algebra $\mathfrak{z}$
satisfying Property~$(*)$ in Definition~\ref{definition:propertystar}.
Put $V_{0,j} := \R X_{0,j}$. Furthermore, thanks to Lemma \ref{lemma:rho-even},
for $(i,j)\in\Lambda$ with $i\neq 0$ we can choose an $SL(2,\R)$-submodule $V_{i,j}$ of $\mathfrak{g}'$ 
isomorphic to $V_{2i+1}$ such that 
\begin{align}
\label{eq:decomposition-g-even}
\mathfrak{g}'=\bigoplus_{(i,j)\in \Lambda} V_{i,j}
(\simeq \bigoplus_{i\in\N} V_{2i+1}^{[\mathfrak{g}':V_{2i+1}]}).
\end{align}

For a discrete subgroup $\Gamma_{g}$ of $SL(2,\mathbb{R})$ isomorphic to $\pi_{1}(\Sigma_{g})$ of genus $g \geq 2$ with 
\begin{align}
g \geq \sum_{i\in\mathbb{N}}  [\mathfrak{g}':V_{2i+1}],
\end{align}
let us construct a continuous family $\rho_{t}\in \Hom(\Gamma_{g},G)$ with $\rho_{0}=\rho|_{\Gamma_{g}}$.
We shall fix 
an injective map 
$f\colon \Lambda \rightarrow 
\{1,2,\ldots,g\}$
and 
a generating set $\{ a_{1},b_{1},\ldots,a_{g},b_{g} \}$ 
of $\Gamma_g$ as in the previous subsection.
In order to apply Lemma \ref{lemma:deform-of-surface-group} (\ref{lemma:deform-of-surface-group-def})
to our setting, for any $(i,j)\in\Lambda$ with $i\neq0$
we can and do 
take a nonzero element $X_{i,j}$ of 
$V_{i,j}(\subset \mathfrak{g})$ satisfying
\begin{align}
\label{eq:Xij-condition}
\mathrm{Ad}(\rho(a_{f(i,j)}))X_{i,j}=X_{i,j}.  
\end{align}
Indeed, recall that $a_{f(i,j)}$ 
is a hyperbolic element of $SL(2,\mathbb{R})$ and that the $\mathfrak{sl}(2,\mathbb{R})$-module
$V_{i,j}$ is of highest weight $2i$ and thus has a vector of weight $0$. Hence there exists a unique
non-zero $X_{i,j}\in V_{i,j}$ satisfying (\ref{eq:Xij-condition}) up to an $\R$-scalar.
Then by applying Lemma \ref{lemma:deform-of-surface-group} (\ref{lemma:deform-of-surface-group-def}) to $\{   X_k \}_{k = 1,\dots,g}$ defined by 
\[
X_{k}:=
\begin{cases}
X_{i,j}  & (\text{if }k = f(i,j) \text{ for } (i,j) \in \Lambda), \\
0 & (\text{otherwise}),
\end{cases}
\]
we obtain the continuous deformation 
$\rho_{t} \colon \Gamma_{g}\rightarrow G$ 
of $\rho|_{\Gamma_g}$
defined by (\ref{def:rho-t}) for $t\in \R$.

To prove Theorem \ref{theorem:deformation-zariski-closure}, 
it suffices to show the following:
\begin{lemma}
    \label{lemma:deformation-zariski-closure}
    In the setting of Theorem \ref{theorem:deformation-zariski-closure},
    let $\rho_{t}$ be as above and fix a sufficiently small positive irrational number $t$. Then the following claims hold:
    \begin{enumerate}
        \item The Zariski closure of $\rho_{t}(\Gamma_{g})$ coincides with $G'$.
        \item Assume that $\mathfrak{g}'$ contains the centralizer $\mathfrak{z}_{\mathfrak{g}}(\rho(\mathfrak{sl}(2,\R)))$.
        Then $\overline{\rho_{t}(\Gamma_{g})}^{\mathrm{fcc}}$ (see Definition \ref{definition:fcc-closure}) contains the Euclidean identity component $G'_{0}$ of $G'$.
    \end{enumerate}    
\end{lemma}

\begin{proof}
Fix any $t \in \R$.
Let us first prove $\overline{\rho_{t}(\Gamma_{g})}^{\Zariski} \subset G'$.
From the decomposition (\ref{eq:decomposition-g-even}),
we see $X_{i,j}\in \mathfrak{g}'$. 
Moreover, recall $\rho(SL(2,\R))\subset G'$ by the assumption (\ref{assumption:SL-G'-Grho}).
Therefore from the construction of $\rho_{t}$ we have $\rho_{t}(\Gamma_{g})\subset G'$ which implies 
$\overline{\rho_{t}(\Gamma_{g})}^{\Zariski} \subset G'$.

Let $R_{t}$ be the Zariski closure $\overline{\rho_{t}(\Gamma_{g})}^{\Zariski}$ of $\rho_{t}(\Gamma_{g})$ in $G$, and we denote by $\mathfrak{r}_{t}$ its real Lie algebra. For a while, our goal is to show $\mathfrak{r}_{t}\supset \mathfrak{g}'$ 
if $t$ is positive irrational and sufficiently small. 
Once this inclusion is established, we can conclude the assertion (i) 
since $G'$ is Zariski-connected.

Now we discuss how small $t>0$ should be for our goal. 
Let $(i,j)\in\Lambda$ with $i\neq 0$.
Since $V_{i,j}(\simeq V_{2i+1})$ is a non-trivial irreducible $\mathfrak{sl}(2,\mathbb{R})$-module, 
we can take $Y_{i,j}\in\rho(\mathfrak{sl}(2,\mathbb{R}))$ with $[X_{i,j},Y_{i,j}]\neq 0$.
We put 
\[
Z_{i,j}(t):=\frac{1}{t}(\mathrm{Ad}(e^{tX_{i,j}})Y_{i,j}-Y_{i,j}).
\]
Then we see $Z_{i,j}(t)\in\mathfrak{r}_{t}$ from the following two properties obtained by applying Lemma \ref{lemma:deform-of-surface-group} (\ref{lemma:deform-of-surface-group-property}) to our construction:
\begin{align}
    \rho(SL(2,\mathbb{R})) \subset R_{t},
    \label{eq:sl2-Lt}\\
    e^{tX_{i,j}} \in R_{t}\text{ for any }(i,j)\in\Lambda.
    \label{eq:Xij-Lt}
\end{align}

Again let $(i,j)\in\Lambda$ with $i\neq 0$.
Denote by
$p_{i,j}\colon\mathfrak{g}'\rightarrow V_{i,j}$
the projection with respect to the decomposition (\ref{eq:decomposition-g-even}).
We take a norm $|\cdot|_{i}$ on the real vector space $V_{2i+1}$ 
for each $i\in\mathbb{N}$ and
fix an isomorphism $q_{i,j}\colon V_{i,j}\simeq V_{2i+1}$
of $SL(2,\mathbb{R})$-modules.
Since $Z_{i,j}(t)=[X_{i,j},Y_{i,j}]+o(1)$ as $t\to0$
and 
$[X_{i,j},Y_{i,j}]\in V_{i,j}$, 
we can and do take 
a sufficiently small $t>0$ 
so that for any $k\neq j$ we have
\begin{align}
    |q_{i,j}(p_{i,j}(Z_{i,j}(t)))|_{i}&>
    \left(1-\frac{1}{[\mathfrak{g}':V_{2i+1}]}\right)
    |q_{i,j}([X_{i,j},Y_{i,j}])|_{i}, \label{ineq:ij-ij}\\
    |q_{i,k}(p_{i,k}(Z_{i,j}(t)))|_{i}&<
    \frac{1}{[\mathfrak{g}':V_{2i+1}]}
    |q_{i,j}([X_{i,j},Y_{i,j}])|_{i}.
    \label{ineq:ik-ij}
\end{align}

Assume that $t$ is positive irrational, and  satisfies the inequalities \eqref{ineq:ij-ij} and \eqref{ineq:ik-ij} for all $(i,j)\in\Lambda$ with $i\neq 0$. Then let us prove $\mathfrak{r}_{t}\supset \mathfrak{g}'$.
It suffices to show that for any $i\in \N$ with $[\mathfrak{g}':V_{2i+1}]>0$ we have
\begin{align}
\label{eq:lt-Vij}
\mathfrak{r}_{t}\supset \bigoplus_{j=1}^{[\mathfrak{g}':V_{2i+1}]} V_{i,j}.
\end{align}
Now we prove (\ref{eq:lt-Vij}) separately for $i=0$ and for $i>0$.
In the case $i=0$, we
    note that $\bigoplus_{j=1}^{[\mathfrak{g}':V_{1}]} V_{0,j}$ is the Lie algebra $\mathfrak{z}=\mathfrak{z}_{\mathfrak{g}'}(\rho(\mathfrak{sl}(2,\R)))$ of the centralizer $Z=Z_{G'}(\rho(SL(2,\R)))$. 
    Here recall that 
    the basis $\{ X_{0,j} \}_j$ of $\mathfrak{z}$ satisfies Property~$(*)$ in Definition~\ref{definition:propertystar}.
    Hence, by applying Lemma \ref{lemma:Zariski-dense-basis-of-AG} (Case 2) to the real  algebraic group $Z$, 
    we obtain $\mathfrak{r}_{t}\supset\mathfrak{z}$
    since the real algebraic subgroup $R_{t}\cap Z$ of $Z$ contains $\exp(tX_{0,j})$ for all $1\leq j\leq [\mathfrak{g}':V_{1}]$ as shown in \eqref{eq:Xij-Lt}. 
    This proves (\ref{eq:lt-Vij}) in the case $i=0$.
    
In the case $i>0$, we recall that
    $\rho(\mathfrak{sl}(2,\R))$ 
    is a subalgebra of the Lie algebra $\mathfrak{r}_{t}$
    as shown in (\ref{eq:sl2-Lt}). 
    Thus the Lie algebra $\mathfrak{r}_{t}$ is an $SL(2,\R)$-submodule of $\mathfrak{g}$ by the adjoint action $\Ad\circ\rho$.
    Since the submodule $\mathfrak{r}_{t}$ is completely reducible, in order to prove (\ref{eq:lt-Vij}) it suffices to show 
    \begin{align}
    \label{eq:dim-hom}
        \dim_{\R}\Hom_{SL(2,\R)}(\mathfrak{r}_{t}, V_{2i+1})
        \geq [\mathfrak{g}':V_{2i+1}].
    \end{align}
    For this purpose, let us prove that the elements
    $q_{i,j}\circ p_{i,j}|_{\mathfrak{r}_{t}}$ 
    ($j=1,\ldots,[\mathfrak{g}':V_{2i+1}]$) of 
    $\operatorname{Hom}_{SL(2,\mathbb{R})}
    (\mathfrak{r}_{t},V_{2i+1})$
    are linearly independent. 
    For a contradiction, suppose that
    there exists
    $(s_{1},\ldots,s_{[\mathfrak{g}':V_{2i+1}]})\in\mathbb{R}^{[\mathfrak{g}':V_{2i+1}]}\smallsetminus\{0\}$
    satisfying
    $\sum_{j=1}^{[\mathfrak{g}':V_{2i+1}]} s_{j}q_{i,j}\circ p_{i,j}|_{\mathfrak{r}_{t}}=0$.
    Without loss of generality, we assume 
    $|s_{1}|\geq |s_{j}|$ for any $j=1,\ldots,[\mathfrak{g}':V_{2i+1}]$.
    In particular, $s_{1}\neq 0$.
    Then, using the inequalities 
    (\ref{ineq:ij-ij}) and (\ref{ineq:ik-ij}),
    we have

    \begin{align*}
        1-\frac{1}{[\mathfrak{g}':V_{2i+1}]}
        <\frac{|q_{i,1}(p_{i,1}(Z_{i,1}(t)))|_{i}}{|q_{i,1}(
        [X_{i,1},Y_{i,1}])|_{i}}  &=\left|\sum_{j=2}^{[\mathfrak{g}':V_{2i+1}]}\frac{s_{j}q_{i,j}(p_{i,j}(Z_{i,1}(t)))}{s_{1}|q_{i,1}(
        [X_{i,1},Y_{i,1}])|_{i}}
        \right|_{i}  \\
        &\leq
        \sum_{j=2}^{[\mathfrak{g}':V_{2i+1}]}
        \frac{|s_{j}q_{i,j}(p_{i,j}(Z_{i,1}(t)))|_{i}}{|s_{1}q_{i,1}(
        [X_{i,1},Y_{i,1}])|_{i}} \\
        &\leq \sum_{j=2}^{[\mathfrak{g}':V_{2i+1}]}
        \frac{|q_{i,j}(p_{i,j}(Z_{i,1}(t)))|_{i}}{|q_{i,1}(
        [X_{i,1},Y_{i,1}])|_{i}}
        \\
        &<1-\frac{1}{[\mathfrak{g}':V_{2i+1}]}.
    \end{align*}
    This is a contradiction, and thus we get
    (\ref{eq:dim-hom}). Hence we also have (\ref{eq:lt-Vij}) in the case (2).
    Therefore we obtain $\mathfrak{r}_{t}\supset \mathfrak{g}'$,
    which implies the assertion (i). 

    Let us prove the assertion (ii). 
    For this purpose, we assume from now on that $\mathfrak{g}'$ contains the centralizer $\mathfrak{z}_{\mathfrak{g}}(\rho(\mathfrak{sl}(2,\R)))$. 
    Fix a positive irrational number $t$ so that the assertion (i) holds.
    Let $L_{t}$ be a closed Lie subgroup of $G$ which has finitely many connected components and contains $\rho_{t}(\Gamma_{g})$.
    We denote by  $\mathfrak{l}_{t}$ and $\overline{\mathfrak{l}_{t}}^{\Zariski}$ the real Lie algebra of $L_{t}$ and that of the Zariski closure $\overline{L_{t}}^{\Zariski}$, respectively. 
    Then the assertion (i) implies that $\overline{\mathfrak{l}_{t}}^{\Zariski}$ contains $\mathfrak{g}'$,
    and in particular, contains $\rho(\mathfrak{sl}(2,\R))$.
    Hence $[\overline{\mathfrak{l}_{t}}^{\Zariski},\overline{\mathfrak{l}_{t}}^{\Zariski}]$ contains all the non-trivial $SL(2,\R)$-submodules of $\overline{\mathfrak{l}_{t}}^{\Zariski}$. 
    Here notice that $\mathfrak{z}_{\mathfrak{g}}(\rho(\mathfrak{sl}(2,\R)))$ is spanned by all $X_{0,j}\in \mathfrak{g}'$ ($1\leq j \leq [\mathfrak{g}':V_{1}]$) since it is contained in $\mathfrak{g}'$.
    Hence the vector space 
    $\overline{\mathfrak{l}_{t}}^{\Zariski}/[\overline{\mathfrak{l}_{t}}^{\Zariski},\overline{\mathfrak{l}_{t}}^{\Zariski}]$ is also spanned by the images $[X_{0,j}]$ of all $X_{0,j} \in \overline{\mathfrak{l}_{t}}^{\Zariski}$.

    Now we consider the Zariski closure $\overline{\mathbf{L}_{t}}^{\Zariski}$ of $L_{t}$ in $\mathbf{G}$ and  
    apply Lemma~\ref{lemma:Zariski-dense-basis-of-AG} (Case~1) to the linear algebraic group  $\mathbf{S}:=\overline{\mathbf{L}_{t}}^{\Zariski}/[\overline{\mathbf{L}_{t}}^{\Zariski},\overline{\mathbf{L}_{t}}^{\Zariski}]$.
    Notice that the subset $\{[X_{0,1}],\ldots, [X_{0,[\mathfrak{g}':V_{1}]}]\}$ of $\overline{\mathfrak{l}_{t}}^{\Zariski}/[\overline{\mathfrak{l}_{t}}^{\Zariski},\overline{\mathfrak{l}_{t}}^{\Zariski}]$ satisfies Property~$(*)$ in Definition~\ref{definition:propertystar}.
    Moreover, each $\exp(tX_{0,j})$ belongs to $L_{t}$ by Lemma~\ref{lemma:deform-of-surface-group} (\ref{lemma:deform-of-surface-group-property}), and 
    the image of $L_{t}$ to $\mathbf{S}(\R)$ is a closed Lie subgroup of $\mathbf{S}(\R)$ with finitely many connected components
    since we have $[\mathfrak{l}_{t},\mathfrak{l}_{t}]=[\overline{\mathfrak{l}_{t}}^{\Zariski},\overline{\mathfrak{l}_{t}}^{\Zariski}]$ by Chevalley's results (Fact \ref{fact:chevalley}).
    Hence Lemma \ref{lemma:Zariski-dense-basis-of-AG} implies $\mathfrak{l}_{t} + [\overline{\mathfrak{l}_{t}}^{\Zariski},\overline{\mathfrak{l}_{t}}^{\Zariski}] = \overline{\mathfrak{l}_{t}}^{\Zariski}$.
    Again by Fact \ref{fact:chevalley}, we obtain 
    $\mathfrak{l}_{t} = \overline{\mathfrak{l}_{t}}^{\Zariski}$.
    This concludes that $\mathfrak{l}_{t}$ contains $\mathfrak{g}'$,
    which proves the assertion (ii). Thus the proof is complete. 
\end{proof}

\section{Non-standard small deformations of standard discontinuous surface groups}
\label{section:non-standard-deformation}

In this section, we give a proof of Theorem \ref{theorem:any_sl2_nonstandard}.
As an application of Theorem \ref{theorem:deformation-zariski-closure}, we study small deformations of standard discontinuous surface groups for a homogeneous space of reductive type.
We use the following theorem to prove Theorems \ref{theorem:any_sl2_nonstandard} and \ref{theorem:ZariskidenseSurface_on_symm}:  
\begin{theorem}\label{theorem:any_sl2_geven}
Let $G/H$ be a homogeneous space of reductive type,
$\rho\colon SL(2,\R) \rightarrow G$ a Lie group homomorphism
such that the $SL(2,\R)$-action on $G/H$ via $\rho$ is proper,
and $\Gamma_{g}$ a discrete subgroup of $SL(2,\mathbb{R})$ isomorphic to $\pi_{1}(\Sigma_{g})$ with genus $g\geq 2$.
Assume 
\[
g\geq \sum_{i\in\mathbb{N}} [\mathfrak{g}:V_{2i+1}].
\]
Then for any neighborhood $W$ of $\rho|_{\Gamma_{g}}$ in $\mathrm{Hom}(\Gamma_{g},G)$
there exists $\rho'\in W$ satisfying the following: 
\begin{enumerate}
    \item 
    \label{property:disconti}
    $\rho'$ is discrete and faithful, and $\rho'(\Gamma_{g})$ is a discontinuous group for $G/H$.
    \item
    \label{property:zariski-closure}
    $\overline{\rho'(\Gamma_{g})}^{\Zariski} = G^\rho_{\even}$.
    \item 
    \label{property:reductive-closure}
    For any closed Lie subgroup $L$ of $G$ with finitely many connected components,
    if $L$ contains $\rho'(\Gamma_{g})$, then $L$ also contains the Euclidean identity component $(G^{\rho}_{\even})_{0}$ of $G^{\rho}_{\even}$.
\end{enumerate}
\end{theorem}

\begin{proof}
    Let $W$ be any neighborhood of $\rho|_{\Gamma_{g}}$ in $\mathrm{Hom}(\Gamma_{g},G)$.
    Now we apply Fact \ref{fact:stability-for-properness} (stability for proper discontinuity) to the setting $L=SL(2,\R)$.
    Then replacing $W$ with a smaller one, we assume that any element of $W$ satisfies the property (\ref{property:disconti}).
    Furthermore applying Theorem \ref{theorem:deformation-zariski-closure} to $G'=G^\rho_{\even}$, 
    we obtain $\rho'\in W$ satisfying the properties (\ref{property:zariski-closure}) and (\ref{property:reductive-closure}) since $\mathfrak{g}^{\rho}_{\even}$ contains the centralizer $\mathfrak{z}_{\mathfrak{g}}(\rho(\mathfrak{sl}(2,\R)))$.
    Thus our claim is proved.
\end{proof}

As an immediate consequence of the properness criterion (Fact \ref{fact:propernesscriterion}) and Lemma \ref{lemma:grhoev_fulrank}, 
we also note the following (the Calabi--Markus phenomenon):

\begin{proposition}
\label{proposition:sl_2_Grho_fulrank}
Let $H$ be a non-compact closed subgroup of a \reductive~$G$.
For any Lie group homomorphism $\rho\colon SL(2,\R)\rightarrow G$, the $(G^\rho_\even)_{0}$-action on $G/H$ is not proper.
\end{proposition}

To prove Theorem \ref{theorem:any_sl2_nonstandard}, it suffices to show the following:
\begin{theorem}
\label{theorem:strong-non-standard}
Let $G/H$, $\rho\colon SL(2,\R)\rightarrow G$, and $\Gamma_{g}$ be as in the setting of Theorem \ref{theorem:any_sl2_geven}.
Assume $H$ is non-compact. Then for any neighborhood $W$ of 
$\rho|_{\Gamma_{g}}$ in 
$\mathrm{Hom}(\Gamma_{g},G)$, there exists a representation $\rho'$ in $W$ satisfying the following: 
\begin{enumerate}
    \item $\rho'$ is discrete and faithful, and $\rho'(\Gamma_{g})$ is a discontinuous group for $G/H$.
    \item For any closed Lie subgroup $L$ of $G$ with finitely many connected components,
    if $L$ contains $\rho'(\Gamma_{g})$, then the $L$-action on $G/H$ is not proper.
\end{enumerate}
\end{theorem}
\begin{proof}
    Our claim follows immediately from Theorem \ref{theorem:any_sl2_geven} and Proposition \ref{proposition:sl_2_Grho_fulrank}.
\end{proof}

\section{Zariski-dense discontinuous surface groups for symmetric spaces}
\label{section:Zariski-dense}

In this section, we give a proof of 
Theorem \ref{theorem:ZariskidenseSurface_on_symm}.

\subsection{Homomorphisms from $SL(2,\R)$ to $G$}
\label{subsection:homSL2G}

In this subsection, in order to prove Theorem \ref{theorem:ZariskidenseSurface_on_symm},
we fix some terminologies and recall some lemmas.

Let $G$ be a real reductive group in the sense of Section \ref{subsection:terminology},
and denote by $\mathfrak{g}$ the Lie algebra of $G$.
Now we introduce the following notation.

\begin{definition}
\label{def:even-hom}
We say that a Lie group homomorphism $\rho\colon SL(2,\R) \rightarrow G$ is even if $G^\rho_{\even} = G$ (see Definition \ref{intro_def:rho-G-g} for the definition of $G^{\rho}_{\even}$).
Furthermore, we also say that a Lie algebra homomorphism $\rho\colon\mathfrak{sl}(2,\R) \rightarrow \mathfrak{g}$
is even if $\mathfrak{g}^\rho_\even = \mathfrak{g}$ (see Definition \ref{definition:geven} 
for the definition of
$\mathfrak{g}^\rho_\even$).
\end{definition}

Since the algebraic group $\mathbf{G}$ is connected, 
one can see that a Lie group homomorphism $\rho\colon SL(2,\R) \rightarrow G$ is even if and only if 
the corresponding Lie algebra homomorphism $\rho\colon \mathfrak{sl}(2,\R) \rightarrow \mathfrak{g}$ 
is even.
We also note that 
the two equivalent conditions 
are also equivalent to the condition that 
the complex nilpotent adjoint orbit $\mathcal{O}^\C_\rho$ is even in the sense of \cite[Chapter 3.8]{CollingwoodMcGovern},
where we define $\mathcal{O}^\C_\rho$ 
by the adjoint orbit 
in the complexification 
$\mathfrak{g}^\C$ of $\mathfrak{g}$ 
through the nilpotent element 
\[
\rho \begin{pmatrix}
    0 & 1 \\ 0 & 0
\end{pmatrix} \in \mathfrak{g} \subset \mathfrak{g}^\C.
\]

Let us fix 
$\theta$, $\mathfrak{a}$ and $\Sigma_+$, and use the symbol $\mathfrak{a}_+$ and $\mathfrak{b}_+$
as in the sense of Section \ref{section:KobayashiBenoist}. 
The proposition and theorem below will play key roles in the next subsection:

\begin{proposition}[{\cite[Lemma 3.1]{Okuda17}}]\label{proposition:sl2b}
    For a Lie algebra homomorphism $\rho\colon \mathfrak{sl}(2,\R) \rightarrow \mathfrak{g}$, 
    if $\rho(A_0) \in \mathfrak{a}_+$, then $\rho(A_0) \in \mathfrak{b}_+$.
\end{proposition}

\begin{theorem}\label{theorem:bspansevensl2}
There exists a family of Lie algebra homomorphisms 
\[
\rho_1,\dots,\rho_k\colon \mathfrak{sl}(2,\R) \rightarrow \mathfrak{g}
\]
satisfying the following conditions:
\begin{enumerate}
    \item $\rho_l$ is even for each $l = 1,\dots,k$.
    \item $\rho_l(A_0) \in \mathfrak{a}_+$ for each $l = 1,\dots,k$.
    \item $\{ \rho_1(A_0),\dots,\rho_k(A_0) \}$ forms a basis of $\mathfrak{b}$.
\end{enumerate}
\end{theorem}

\begin{proof}[of Theorem \ref{theorem:bspansevensl2}]
For the proof we only need to focus on the case where $\mathfrak{g}$ is semisimple 
since $\mathfrak{b}$ is contained in the semisimple part of $\mathfrak{g}$ (see Section \ref{section:KobayashiBenoist}).
For semisimple $\mathfrak{g}$, the claim is proved in \cite[Theorem 1.1 and Remark 3.2]{Okuda17}.
\end{proof}

\subsection{Proof of Theorem  \ref{theorem:ZariskidenseSurface_on_symm}}
\label{section:proof_ZSonSymm}

Let us take $(G,H)$ as in Section \ref{subsection:terminology} and assume that $(G,H)$ a symmetric pair. 
We denote by $(\mathfrak{g},\mathfrak{h})$ the symmetric pair of Lie algebras corresponding to $(G,H)$.
Let us take a Cartan involution $\theta$ of $G$ with $\theta(H) \subset H$. The Cartan decompositions of $\mathfrak{g}$ and $\mathfrak{h}$ corresponding to $\theta$ are denoted by 
\begin{align*}
    \mathfrak{g} = \mathfrak{k} + \mathfrak{p} \text{ and } ~
    \mathfrak{h} = \mathfrak{k}_{\mathfrak{h}} + \mathfrak{p}_{\mathfrak{h}}.
\end{align*}

Let us fix a maximal abelian subspace $\mathfrak{a}_{\mathfrak{h}}$ of $\mathfrak{p}_\mathfrak{h}$ and that $\mathfrak{a}$ of $\mathfrak{p}$ with $\mathfrak{a}_{\mathfrak{h}} \subset \mathfrak{a}$. The restricted root system of $(\mathfrak{g},\mathfrak{a})$ is denoted by $\Sigma$.
We shall take an ordered basis $\mathcal{B}_{\mathfrak{h}}$ of $\mathfrak{a}_{\mathfrak{h}}$ 
and extend it to that $\mathcal{B}_{\mathfrak{g}}$ of $\mathfrak{a}$,
and write $\Sigma_+$ for the positive system of $\Sigma$ with respect to the ordered basis $\mathcal{B}_{\mathfrak{g}}$.
The closed Weyl chamber in $\mathfrak{a}$ with respect to $\Sigma_+$ is written as $\mathfrak{a}_+$.
As in Section \ref{section:KobayashiBenoist}, 
we define the subspace $\mathfrak{b}$ of $\mathfrak{a}$ by 
\[
\mathfrak{b} := \{ A \in \mathfrak{a} \mid \iota(A) = A \}
\]
and put $\mathfrak{b}_+ := \mathfrak{b} \cap \mathfrak{a}_+$, 
where $\iota := -w_0$ denotes the opposition involution with respect to the positive system $\Sigma_+$. 

The goal of this section is to show the following theorem, which includes Theorem~\ref{theorem:ZariskidenseSurface_on_symm}:

\begin{theorem}\label{theorem:genZariskidenseSurface_on_symm}
In the setting above, 
the following eight conditions are equivalent:
\begin{enumerate}
\item \label{item:Zariskidense:properSL2} 
There exists a Lie group homomorphism $\rho\colon SL(2,\R) \rightarrow G$ such
that $SL(2,\R)$ acts on $G/H$ properly via $\rho$.
\item \label{item:Zariskidense:properevenSL2} There exists an even Lie group homomorphism $\rho\colon SL(2,\R) \rightarrow G$ such
that $SL(2,\R)$ acts on $G/H$ properly via $\rho$.
\item \label{item:Zariskidense:Zariskidense} There exist an integer $g\geq 2$ and a discontinuous group $\Gamma_g$ for $G/H$ such that $\Gamma_g$ is isomorphic to $\pi_1(\Sigma_g)$ and Zariski-dense in $G$.
\item \label{item:Zariskidense:ZariskidenseFree} There exist an integer $g \geq 2$ and a discontinuous group $F_g$ for $G/H$ such that $F_g$ is the free group of rank $g$ and Zariski-dense in $G$.
\item \label{item:Zariskidense:nonva_disconti} 
There exists a discontinuous group $\Gamma$ for $G/H$ which is not virtually abelian.
\item \label{item:Zariskidense:ba} $\mathfrak{b}_+ \not \subset \mathfrak{a}_{\mathfrak{h}}$.
\item \label{item:Zariskidense:LieSL2} There exists a Lie algebra homomorphism $\rho\colon \mathfrak{sl}(2,\R) \rightarrow \mathfrak{g}$ such that 
$\rho(A_0) \in \mathfrak{a}_+ \smallsetminus \mathfrak{a}_{\mathfrak{h}}$.
\item \label{item:Zariskidense:evenLieSL2} There exists an even Lie algebra homomorphism $\rho\colon \mathfrak{sl}(2,\R) \rightarrow \mathfrak{g}$ such that 
$\rho(A_0) \in \mathfrak{a}_+ \smallsetminus \mathfrak{a}_{\mathfrak{h}}$.
\end{enumerate}
\end{theorem}

In our proof of Theorem \ref{theorem:genZariskidenseSurface_on_symm}, 
we apply the following lemma:

\begin{lemma}\label{lemma:positivity}
In the setting above, 
\[
\mathfrak{a}_+ \cap W \cdot \mathfrak{a}_{\mathfrak{h}} = \mathfrak{a}_+ \cap \mathfrak{a}_{\mathfrak{h}}.
\]
In particular, $\mathfrak{b}_+ \not \subset W \mathfrak{a}_\mathfrak{h}$ if and only if $\mathfrak{b}_+ \not \subset \mathfrak{a}_\mathfrak{h}$.
\end{lemma}

\begin{proof}[of Lemma \ref{lemma:positivity}]
    Our goal is to show that 
    \[
    \mathfrak{a}_+ \cap W \cdot \mathfrak{a}_\mathfrak{h} \subset \mathfrak{a}_\mathfrak{h}.
    \]
    Without loss of the generality, one can assume that $\mathfrak{g}$ is semisimple.
    By applying the arguments in \cite[Sections 2 and 3]{OshimaSekiguchi84} for the associated pair of  $(\mathfrak{g},\mathfrak{h})$, 
    we see that 
    \[
    \Sigma_\mathfrak{h} := \{ \alpha|_{\mathfrak{a}_\mathfrak{h}} \in \mathfrak{a}_\mathfrak{h}^\vee \mid \alpha \in \Sigma \} \smallsetminus \{ 0 \}
    \] 
    forms a root system in $\mathfrak{a}_{\mathfrak{h}}^\vee$.
    We write $W_{\mathfrak{h}}$ for the Weyl group of $\Sigma_\mathfrak{h}$ acting on $\mathfrak{a}_{\mathfrak{h}}$,
    and define the subgroup $W'$ of $W$ by 
    \[
    W' := \{ w \in W \mid w \mathfrak{a}_{\mathfrak{h}} \subset \mathfrak{a}_{\mathfrak{h}} \}.
    \]
    Then by applying \cite[Lemma (7.2)(ii)]{OshimaSekiguchi84} for the associated pair $(\mathfrak{g},\mathfrak{h}^{a})$, we have  \begin{align}
    W'|_{\mathfrak{a}_{\mathfrak{h}}} = W_\mathfrak{h}. \label{eq:W'Wh}
    \end{align}
    Let us denote by $(\Sigma_{\mathfrak{h}})_+$ the positive system of $\Sigma_\mathfrak{h}$ with respect to the ordered basis $\mathcal{B}_{\mathfrak{h}}$ of $\mathfrak{a}_{\mathfrak{h}}$.
    Then 
    \[
    (\Sigma_{\mathfrak{h}})_+ = \{ \alpha|_{\mathfrak{a}_\mathfrak{h}} \mid \alpha \in \Sigma_+ \} \smallsetminus \{ 0 \}.
    \]
    We put  
    \[
    (\mathfrak{a}_{\mathfrak{h}})_+ := \{ A \in \mathfrak{a}_{\mathfrak{h}} \mid \langle A, \xi \rangle \geq 0 \text{ for any } \xi \in (\Sigma_{\mathfrak{h}})_+ \}.
    \] 
    Then $(\mathfrak{a}_{\mathfrak{h}})_+$ is a fundamental domain of the $W_{\mathfrak{h}}$-action on $\mathfrak{a}_{\mathfrak{h}}$, and 
    \begin{align}
        (\mathfrak{a}_{\mathfrak{h}})_+ = \mathfrak{a}_+ \cap \mathfrak{a}_{\mathfrak{h}}. \label{eq:ahpositive}
    \end{align}
    
    Let us take any $A \in \mathfrak{a}_{\mathfrak{h}}$.
    Since $\mathfrak{a}_+$ is a fundamental domain of the $W$-action on $\mathfrak{a}$, 
    there uniquely exists $A' \in \mathfrak{a}_+ \cap W \cdot A$.
    We only need to show that $A' \in \mathfrak{a}_\mathfrak{h}$.
    Recall that $(\mathfrak{a}_{\mathfrak{h}})_+$ is also a fundamental domain of the $W_{\mathfrak{h}}$-action on $\mathfrak{a}_{\mathfrak{h}}$, 
    one can find $A'' \in (\mathfrak{a}_{\mathfrak{h}})_+ \cap W_{\mathfrak{h}} \cdot A$.
    By \eqref{eq:W'Wh} and \eqref{eq:ahpositive}, 
    we see that $A'' \in \mathfrak{a}_+ \cap W \cdot A$.
    By the uniqueness of $A'$ above, we have $A' = A''$, 
    and hence $A' \in \mathfrak{a}_\mathfrak{h}$.
\end{proof}

Let us give a proof of 
Theorem \ref{theorem:genZariskidenseSurface_on_symm}
as follows:

\begin{proof}[of Theorem  \ref{theorem:genZariskidenseSurface_on_symm}]
Our strategy for the proof is the following:
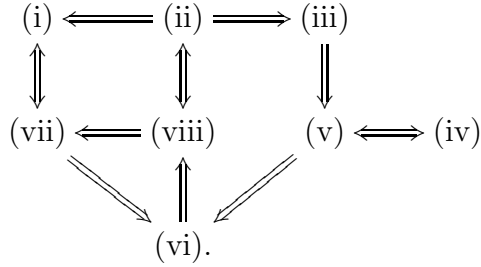
\[
\xymatrix{
\eqref{item:Zariskidense:properSL2} \ar@{<=>}[d] \ar@{<=}[r] & 
\eqref{item:Zariskidense:properevenSL2} \ar@{<=>}[d] \ar@{=>}[r] & 
\eqref{item:Zariskidense:Zariskidense} \ar@{=>}[d] \\
 \eqref{item:Zariskidense:LieSL2} \ar@{=>}[rd] \ar@{<=}[r] & 
 \eqref{item:Zariskidense:evenLieSL2} \ar@{<=}[d] & 
 \eqref{item:Zariskidense:nonva_disconti} \ar@{=>}[ld] \ar@{<=>}[r] & \eqref{item:Zariskidense:ZariskidenseFree} \\
 & \eqref{item:Zariskidense:ba}.
}
\]

The implications 
\eqref{item:Zariskidense:properevenSL2}
$\Rightarrow$
\eqref{item:Zariskidense:properSL2} 
and 
\eqref{item:Zariskidense:evenLieSL2}
$\Rightarrow$
\eqref{item:Zariskidense:LieSL2}
are both trivial.
The equivalences 
\eqref{item:Zariskidense:properSL2} 
$\Leftrightarrow$
\eqref{item:Zariskidense:LieSL2}
and 
\eqref{item:Zariskidense:evenLieSL2}
$\Leftrightarrow$
\eqref{item:Zariskidense:properevenSL2}
follows from Corollary  \ref{cor:sl_2-proper} and Lemma \ref{lemma:positivity}.
The equivalence 
 \eqref{item:Zariskidense:ZariskidenseFree}
$\Leftrightarrow$ 
 \eqref{item:Zariskidense:nonva_disconti}
comes from Fact \ref{fact:Benoist}.
As a direct consequence of Theorem \ref{theorem:any_sl2_geven},  
we have 
\eqref{item:Zariskidense:properevenSL2} $\Rightarrow$ \eqref{item:Zariskidense:Zariskidense}.
The implication
\eqref{item:Zariskidense:Zariskidense}
$\Rightarrow$
\eqref{item:Zariskidense:nonva_disconti}
follows from the fact that the surface group $\pi_1(\Sigma_g)$ is not virtually abelian.
Proposition \ref{proposition:sl2b} shows 
\eqref{item:Zariskidense:LieSL2}
$\Rightarrow$
\eqref{item:Zariskidense:ba}.
By combining Fact \ref{fact:Benoist} with Lemma \ref{lemma:positivity}, we also have 
\eqref{item:Zariskidense:nonva_disconti}
$\Rightarrow$
\eqref{item:Zariskidense:ba}.
Finally, the implication 
\eqref{item:Zariskidense:ba}
$\Rightarrow$
\eqref{item:Zariskidense:evenLieSL2}
follows from 
Theorem \ref{theorem:bspansevensl2}.
\end{proof}

\subsection{Remarks on non-symmetric cases}

In this subsection, we construct a homogeneous space $G/H$ of reductive type with simple $G$ satisfying the following two conditions:
\begin{enumerate}
    \item $G/H$ admits some proper (non-even) $SL(2,\R)$-actions, but 
    \item $G/H$ does not admit proper even $SL(2,\R)$-actions.
\end{enumerate}
Recall that for symmetric $G/H$, 
such examples do not exist by 
Theorem \ref{theorem:genZariskidenseSurface_on_symm}.

We take $G$ as $SL(5,\R)$ and 
$H$ the analytic subgroup corresponding to 
the abelian subalgebra $\mathfrak{h}$ of $\mathfrak{g}$ defined by 
\[
\mathfrak{h} := \mathrm{span}_{\R} \{ \diag(2,-2,0,0,0), \diag(4,2,0,-2,-4) \}.
\]

    \begin{tabular}{c|c|c}
    \hline 
    Symbol & even or non-even & $\rho(A_0) \in \mathfrak{a}_+$ \\
    \hline \hline 
    $[5]$ & even & $\diag(4,2,0,-2,-4)$ \\
    $[4,1]$ & non-even & $\diag(3,1,0,-1,-3)$ \\
    $[3,2]$ & non-even & $\diag(2,1,0,-1,-2)$ \\
    $[3,1^2]$ & even & $\diag(2,0,0,0,-2)$ \\
    $[2^2,1]$ & non-even & $\diag(1,1,0,-1,-1)$ \\
    $[2,1^3]$ & non-even & $\diag(1,0,0,0,-1)$ \\ 
    \hline
    \end{tabular}

There exist only two types of even homomorphisms $\rho\colon \mathfrak{sl}(2,\R) \rightarrow \mathfrak{sl}(5,\R)$ corresponding to the partitions $[5]$ and $[3,1^2]$.
Since 
\begin{align*}
    \diag(4,2,0,-2,-4), \diag(2,0,0,0,-2) \in W \cdot \mathfrak{a}_{\mathfrak{h}}, 
\end{align*}
the $SL(2,\R)$-actions on $G/H$ via $[5]$ and $[3,1^2]$ are both non-proper.
Therefore, there do not exist proper even $SL(2,\R)$-actions on $G/H$ in this case.

One can see that 
\begin{align*}
    \diag(3,1,0,-1,-3), \diag(1,1,0,-1,-1) \not \in W \cdot \mathfrak{a}_{\mathfrak{h}}.
\end{align*}
Thus $G/H$ admits two types of proper non-even $SL(2,\R)$-actions via Lie algebra homomorphisms corresponding to $[4,1]$ and $[2^2,1]$.

The example of $G/H$ above 
gives a counterexample of Theorem \ref{theorem:genZariskidenseSurface_on_symm} 
for the case where $G/H$ is not a symmetric space.
However, this does not mean that $G/H$ gives a counterexample of Theorem \ref{theorem:ZariskidenseSurface_on_symm}.
In fact, the following holds: 

\begin{proposition}
Zariski-dense discontinuous surface groups for the above $G/H$ exist.
\end{proposition}

\begin{proof}
Let $\rho\colon SL(2,\R)\rightarrow G$ be one of the two non-even Lie group homomorphisms corresponding to $[4,1]$ and $[2^2,1]$, 
and $\Gamma_{g}$ a discrete subgroup of $SL(2,\R)$ isomorphic to the surface group $\pi_{1}(\Sigma_{g})$
of genus $g$. Assume $g \geq 2(\dim G)^{2}$.
Then by combining a result of Kim--Pansu (\cite[Theorem 1]{KIMPAN15}) with Fact \ref{fact:stability-for-properness} (stability for proper discontinuity), we obtain $\rho'\in \Hom(\Gamma_{g},G)$ sufficiently close to $\rho|_{\Gamma_{g}}$ such that the $\Gamma_{g}$-action on $G/H$ via $\rho'$ is properly discontinuous and that $\rho'(\Gamma_{g})$ is Zariski-dense in $G$.
\end{proof}

We also note that, related to Theorem  \ref{theorem:ZariskidenseSurface_on_symm}, 
the following problems are still open:
\begin{problem}
    Let $G/H$ be a homogeneous space of reductive type:
    \begin{enumerate}
        \item
        If there exists a discontinuous group $\Gamma$ for $G/H$ isomorphic to a surface group,
        then can we choose $\Gamma$ to be Zariski-dense in $G$?
        \item
        If there exists a non virtually abelian discontinuous group $\Gamma$ for $G/H$,
        then can we choose $\Gamma$ to be isomorphic to a surface group and Zariski-dense in $G$?
    \end{enumerate}
\end{problem}

\section{Example}
\label{section:examples}

In this section, as an example of Theorem~\ref{theorem:ZariskidenseSurface_on_symm}, 
let us consider the classical symmetric space $G/H:=SU(p,q)/U(p,q-1)$ for $p\geq q$. 
We explicitly give two Lie group homomorphisms from $SL(2,\R)$ to $G$ which induce proper $SL(2,\R)$-actions on $G/H$, 
determine whether the two homomorphisms are even or non-even in the sense of Definition~\ref{def:even-hom},
and apply Theorem~\ref{theorem:any_sl2_geven}.
Note that if $p<q$, then no infinite discontinuous groups  exist (the Calabi--Markus phenomenon). 

To apply the properness criterion (Corollary \ref{cor:sl_2-proper}), 
let us realize the symmetric pair $(SU(p,q),U(p,q-1))$
so that their maximal split abelian subspaces have simple forms.
For $k\in \N$, we denote by $I_k$ and $O_k$ the identity matrix and the zero matrix of size $k$, respectively. 
For non-negative integers $p\geq q$, we consider the matrix 
\[
B_{p,q}
:=\begin{pmatrix} 
 & & I'_q\\
 & I_{p-q} & \\
 I'_q & & 
\end{pmatrix}
\text{ where } 
I'_q=\begin{pmatrix}
& & 1\\
& \iddots & \\
1& & \end{pmatrix}\in M(q,\R).
\]
Let $\{e_1,\ldots, e_n\}$ be the standard basis of $\C^n$ ($n=p+q$).
Note that the matrix $B_{p,q}$ defines a hermitian form $h_{p,q}$ of signature $(p,q)$ on $\C^n$ 
with $h_{p,q}(e_1-e_n,e_1-e_n)<0$. 
We realize the symmetric pair $(SU(p,q),U(p,q-1))$ as the following pair $(G,H)$:
\begin{align*}
    G&:=\{g\in SL(p+q,\C) \mid g^{*}B_{p,q} g= B_{p,q}\}, \\
    H&:=\{ g\in G \mid g (e_1-e_n)\in \C (e_1-e_n)\},  
\end{align*}
where $g^*:=\trans \bar{g}$ for a complex matrix $g$. 
Then the corresponding Lie algebras are given as follows:
\begin{align*}
    \mathfrak{g}&=\{ X\in \mathfrak{sl}(p+q,\C)\mid X^{*}B_{p,q}+B_{p,q} X=0\},\\
    \mathfrak{h}&=\{ X\in \mathfrak{g}\mid  X(e_1-e_n) \in \C(e_1-e_n)\}. 
\end{align*}
We can take maximal split abelian subspaces of $\mathfrak{g}$ and $\mathfrak{h}$ as follows:
\begin{align*}
    \mathfrak{a}&=\{ \diag(a_1,\ldots, a_q,\underbrace{0,\ldots, 0}_{p-q}, -a_q,\ldots, -a_1) \in \mathfrak{sl}(p+q,\R) \}, \\
    \mathfrak{a}_{\mathfrak{h}}&=
    \{  \diag(a_1,\ldots, a_q,0, \ldots, 0,-a_q,\ldots, -a_1)\in \mathfrak{a}\mid a_1=0\}.
\end{align*}
The Weyl group $W=W(\mathfrak{g},\mathfrak{a})$ is the group of permutations and sign changes of $\{a_1,\ldots, a_q\}$. 

Let $\phi\colon \mathfrak{g}\rightarrow \mathfrak{sl}(p+q,\C)$ be the natural inclusion.
Now we explicitly define two Lie algebra homomorphisms 
$\rho_{1},\rho_{2}\colon\mathfrak{sl}(2,\R)\rightarrow \mathfrak{g}(\simeq \mathfrak{su}(p,q))$ 
such that the representations $\phi\circ\rho_{1},\phi\circ\rho_{2}$ of $\mathfrak{sl}(2,\R)$ 
correspond to the partitions $[2^{q},1^{p-q}],[2q+1,1^{p-q-1}]$, respectively.
Let $\{h,e,f\}$ be the standard $\mathfrak{sl}(2,\R)$-triple:
\[
h=A_{0}=\begin{pmatrix}1& 0\\ 0 & -1\end{pmatrix},\ 
e=\begin{pmatrix}0& 1\\ 0 & 0\end{pmatrix},\ 
f=\begin{pmatrix} 0 & 0 \\ 1& 0\end{pmatrix}.
\]
Then we put
\begin{align*}
   \rho_1(h)&:=\diag(\underbrace{1,1,\ldots,1}_{q}, \underbrace{0,\ldots, 0}_{p-q}, \underbrace{-1, \ldots, -1}_{q}),\\
   \rho_1(e)&:=
\sqrt{-1}\begin{pmatrix} 
 & & I_q\\
 & O_{p-q} & \\
 O_{q} & & 
\end{pmatrix},\ \rho_1(f):=\rho_1(e)^*.
\end{align*}
When $p\geq q+1$, we put
\begin{align*}
    \rho_2(h)&:=\diag(2q, 2q-2, \ldots,2,\underbrace{0,\ldots, 0}_{p-q}, -2 ,\ldots,  -2q),\\
    \rho_2(e)&:=
    \begin{pmatrix}
     e(c_1,\ldots, c_q) &  &  \\
                        & O_{p-q-1} &  \\
                        &           & e(c_{q-1},\ldots, c_{1})
    \end{pmatrix}+c_{q} E_{q+1,p+1},\\
    \rho_2(f)&:=\rho_2(e)^*, 
\end{align*}
where $c_k=\sqrt{-1}\sqrt{k(2q+1-k)}$ ($k=1,\ldots, q$), $E_{i,j}$ is the matrix unit, and 
\[ e(d_1,\ldots, d_{m-1}):=
\begin{pmatrix}
 0 &d_1 &   &  \\  
  & \ddots& \ddots  \\
  & & \ddots & d_{m-1}\\
  & &       & 0
\end{pmatrix} \in M(m,\C) \text{ for } m\in\N. 
\]
Then we can check that the matrices $\rho_{i}(h),\rho_{i}(e),\rho_{i}(f)$ belong to $\mathfrak{g}$ for $i=1,2$
and that $\rho_{1}$ and $\rho_{2}$ define Lie algebra homomorphisms from $\mathfrak{sl}(2,\R)$ to $\mathfrak{g}$.
Let us denote by the same symbols the corresponding Lie group homomorphisms from $SL(2,\R)$ to $G$.

\begin{claim}
    Let $G/H$ be as above and $\rho\colon SL(2,\R)\rightarrow G$ one of the two Lie group homomorphisms $\rho_{1}$ and $\rho_{2}$.
    The following assertions hold:
    \begin{enumerate}
    \item
    The $SL(2,\R)$-action on $G/H$ via $\rho$ is proper.
    \item
    We have
    \[
    G^{\rho}_{\even} = 
    \begin{cases}
      S(U(q,q)\times U(p-q)) & (\rho=\rho_{1}), \\
      G & (p\geq q+1,\ \rho=\rho_{2}).
    \end{cases}
    \]
    In particular, we get Table~\ref{table:g_rho_even} which shows whether the Lie group homomorphism $\rho$ is even or non-even.
    \begin{table}[!h]
    \centering
    \caption{even or non-even}
    \begin{tabular}{c|c|c}
        \hline
        $\rho$ &  $p\geq q+1$ & $p=q$ \\
        \hline
        \hline
        $\rho_{1}$ & non-even & even \\
        $\rho_{2}$ & even & undefined \\
        \hline
    \end{tabular}
    \label{table:g_rho_even}
    \end{table}

    \item
    Let $\Gamma_{g}$ be a discrete subgroup of $SL(2,\R)$ isomorphic to the surface group $\pi_{1}(\Sigma_{g})$ of genus $g$.
    If 
    \[
    g\geq \begin{cases}
    2q^{2}+(p-q)^2-1 & (\rho=\rho_{1}), \\
    (p-q)^2+2q-1 & (\rho=\rho_{2}),
    \end{cases}
    \]
    then there exists a discrete and faithful representation $\rho'\in \Hom(\Gamma_g, G)$ sufficiently close to $\rho|_{\Gamma_g}$ such that $\rho'(\Gamma_g)$ is a discontinuous group for $G/H$ whose Zariski closure in $G$ coincides with $G^{\rho}_{\even}$.
    \end{enumerate}
\end{claim}

\begin{proof}
Since we have $\rho_{i}(h)=\rho_{i}(A_0)\not \in W\mathfrak{a}_\mathfrak{h}$ for each $i=1,2$, 
it follows from Corollary~\ref{cor:sl_2-proper} that the $SL(2,\R)$-action on $G/H$ via each $\rho_{i}$ is proper,
which proves the assertion (i). 

The assertion (ii) follows from the formula below 
(see Definition~\ref{intro_def:rho-G-g} for $\sigma(\rho)$):
\[
    \sigma(\rho_1) = 
       \begin{pmatrix}
       -I_{q} &  &  \\
        & I_{p-q} &  \\ 
        &  & -I_{q} 
       \end{pmatrix},\ 
    \sigma(\rho_2) = I_{p+q}.
\]

The assertion (iii) follows from Theorem~\ref{theorem:any_sl2_geven} and Lemma~\ref{lem:lower-bound-of-genus}.
\end{proof}

Finally let us remark the following in the case $p=q$:
\begin{claim}
If $p=q$, the homogeneous space $G/H\simeq SU(p,p)/U(p,p-1)$ admits no proper non-even $SL(2,\R)$-actions. 
\end{claim}
\begin{proof}
Let $\rho\colon SL(2,\R)\rightarrow G$ be a Lie group homomorphism 
such that the $SL(2,\R)$-action on $G/H$ via $\rho$ is proper. 
Then we claim that no odd-dimensional irreducible representations of $SL(2,\R)$ occur in the $2p$-dimensional representation $\phi\circ\rho$,
where $\phi\colon G\rightarrow GL(2p,\C)$ be the natural inclusion.
Assume otherwise for a contradiction.
Then at least two odd-dimensional irreducible representations occur in $\phi\circ\rho$
and thus the zero-eigenspace of $\phi(\rho(A_{0}))$ is
at least two-dimensional.
Replacing $\rho$ by some $G$-conjugate, we assume $\rho(A_{0})\in\mathfrak{a}$.
Then $\rho(A_0)$ is contained in $W\mathfrak{a}_\mathfrak{h}$, which contradicts the properness of 
the $SL(2,\R)$-action on $G/H$ via $\rho$ by Corollary~\ref{cor:sl_2-proper}. Hence no odd-dimensional irreducible representations occur in  $\phi\circ\rho$. Then, we have 
$\sigma(\rho)=-I_{2p}$, and thus $G^\rho_{\even}=G$. 
This proves that $\rho$ is even. 
\end{proof} 

\appendix
\section{Some remarks for Fact \ref{fact:Benoist}}
Fact \ref{fact:Benoist} is established by 
Benoist \cite{Benoist96} for semisimple $G$.
In this appendix, 
for the sake of the completeness, 
we show that  
Fact \ref{fact:Benoist} 
also holds for the case where $G$ is reductive, 
by giving some minor modifications of the arguments in \cite[Section 7]{Benoist96}.

We follow the setting and notation in Section \ref{section:KobayashiBenoist}. 
Noting that $\mathfrak{b}$ is contained in $\mathfrak{a}\cap \mathfrak{g}_{ss}$,
one can prove Fact \ref{fact:Benoist} by combined an argument similar to the proof of \cite[Th\'{e}or\`{e}me 7.5]{Benoist96} with the following two facts: 

\begin{fact}
\label{fact:Benoist_free_necessary}
Let $\Gamma$ be a discrete subgroup of a \reductive~$G$.
If 
\[
\mu(\Gamma) \pitchfork \mathfrak{b}_+ \text{ in } \mathfrak{a}
\]
holds, 
then $\Gamma$ should be virtually abelian,
that is, $\Gamma$ has a finite-index abelian normal subgroup.
\end{fact}

\begin{fact}
\label{fact:Benoist_Cartan_cone}
Assume that a \reductive~$G$ has a non-compact simple factor.
Let us fix any open convex cone $\Omega$ in $\mathfrak{a}_+$ with $\iota(\Omega) \subset \Omega$.
Then there exists a Zariski-dense discrete subgroup $\Gamma$ of $G$ such that 
$\Gamma$ is a non-abelian free group
and $\mu(\Gamma) \subset \Omega \cup \{ 0 \}$.
\end{fact}

Fact~\ref{fact:Benoist_free_necessary} in semisimple cases 
can be found in \cite[Th\'{e}or\`{e}me 3.3]{Benoist96} 
and the proof works even for reductive cases.

For semisimple cases, 
Fact~\ref{fact:Benoist_Cartan_cone} 
corresponds exactly to {\cite[Proposition 1.6]{Benoist96}}.
In the rest of this appendix, we focus on Fact~\ref{fact:Benoist_Cartan_cone} in a general reductive situation.

Let us consider the setting of Fact~\ref{fact:Benoist_Cartan_cone}. 
By applying the arguments in \cite[Section 7]{Benoist96} to the semisimple part of $G$, 
one can construct a Zariski-dense discrete free subgroup  $\Gamma$ of the semisimple part.
As we discuss below, 
by considering a small deformation of $\Gamma$ in $G$, 
Fact~\ref{fact:Benoist_Cartan_cone} is reduced to the semisimple case. 

We shall recall the definition of the \emph{Lyapunov projection}
$\lambda\colon G \rightarrow \mathfrak{a}_{+}$ as in \cite[Section 7.1]{Benoist96}. For $g\in G$ let $g=g_{e}g_{h}g_{u}$ be its Jordan decomposition, where $g_{e}$, $g_{h}$, and $g_{u}$ are mutually commutative elements of $G$, 
and are elliptic, hyperbolic, and unipotent, respectively. Then $\lambda(g)$ is a unique element of $\mathfrak{a}_{+}$ 
such that $\exp(\lambda(g))$ is $G$-conjugate to $g_{h}$. Notice $\lambda(g)=\mu(g)$ if $g$ is central in $G$.

For an integer $t\geq 2$, let $F_{t}$ be a free group of rank $t$ on generators $\gamma_{1},\ldots,\gamma_{t}$,
and we put $E_{t}:=\{\gamma_{1},\gamma_{1}^{-1},\ldots,\gamma_{t},\gamma_{t}^{-1}\}$.
To prove Fact \ref{fact:Benoist_Cartan_cone}, we show the following in our setting:
\begin{lemma}
\label{lemma:Benoist-Schottky}
    Assume that a \reductive~$G$ has a non-compact simple factor.
    For any open convex cone $\Omega$ in $\mathfrak{a}_+$ which is stable under the opposition involution $\iota$,
    there exist a compact neighborhood $U$ of $0$ in $\mathfrak{a}$, an integer $t\geq 2$, and 
    a group homomorphism $\varphi\colon F_{t}\rightarrow G$
    satisfying the following properties: 
    \begin{enumerate}
    \item
    $\varphi$ is injective.
    \item
    $\varphi(F_{t})$ is Zariski-dense in $G$.
    \item
    For any $g\in E_{t}$ we have $\lambda(\varphi(g))+U\subset \Omega$.
    \item
    For any reduced word $g=g_{l}\cdots g_{1}$ ($g_{1},\ldots,g_{l}\in E_{t}$), we have
    \[
    \mu(\varphi(g)) \in \sum_{i=1}^{l} (\lambda(\varphi(g_{i}))+U).
    \]
    \end{enumerate}
\end{lemma}

\begin{proof}
    For semisimple $G$, we see immediately that the claim holds 
    by combining the argument in the proof of \cite[Corollaire 7.3]{Benoist96} with \cite[Lemme 7.2]{Benoist96}.

    For reductive $G$ with non-discrete center, let $\mathbf{G}_{ss}$ be the semisimple part 
    of $\mathbf{G}$ and $\mathbf{Z}(\mathbf{G})_{0}$ the identity component of the center of $\mathbf{G}$.
    We denote by $\mathfrak{g}_{ss}$ and $\mathfrak{z}(\mathfrak{g})$ the corresponding real Lie algebras.
    Applying our claim to the semisimple Lie group $G_{ss}:=G\cap \mathbf{G}_{ss}(\C)$ and to
    the open convex cone $\Omega\cap \mathfrak{g}_{ss}$ in $\mathfrak{a}_+\cap \mathfrak{g}_{ss}$, 
    we get a compact neighborhood $U_{ss}$ of $0$ in $\mathfrak{a}\cap \mathfrak{g}_{ss}$ and a group 
    homomorphism $\varphi_{ss}\colon F_{t} \rightarrow G_{ss}$ for some $t\geq 2$ satisfying the following properties:
    \begin{enumerate}
    \item [(i)$_{ss}$]
    $\varphi_{ss}$ is injective.
    \item[(ii)$_{ss}$]
    $\varphi_{ss}(F_{t})$ is Zariski-dense in $G_{ss}$.
    \item[(iii)$_{ss}$]
    For any $g\in E_{t}$ we have $\lambda(\varphi_{ss}(g))+U_{ss} \subset \Omega$.
    \item[(iv)$_{ss}$]
    For any reduced word $g=g_{l}\cdots g_{1}$ ($g_{1},\ldots,g_{l}\in E$), we have
    \[
    \mu(\varphi_{ss}(g)) \in \sum_{i=1}^{l}(\lambda(\varphi_{ss}(g_{i}))+U_{ss}).
    \]
    \end{enumerate}
    
    Take a compact neighborhood $U_{z}$ of $0$ in $\mathfrak{a} \cap \mathfrak{z}(\mathfrak{g})$
    such that $\lambda(\varphi_{ss}(g)) + U_{ss} + U_{z} \subset \Omega$ for any $g\in E$.
    Notice that $U:=U_{ss} + U_{z}$ is a compact neighborhood  of $0$ in $\mathfrak{a}$. 
    We choose a non-torsion central element $a\in G$ from the nonempty open subset $\lambda^{-1}(U_{z})\cap Z(G)$ of the center $Z(G)$ that generates a Zariski-dense subgroup of $\mathbf{Z}(\mathbf{G})_{0}$.
    Let us define a group homomorphism $\varphi\colon F_{t}\rightarrow G$ by $\varphi(\gamma_{i}):=a\varphi_{ss}(\gamma_{i})$ for each $i=1,\ldots,t$. 
    
    Now we prove that the compact neighborhood $U$ and the group homomorphism $\varphi$ satisfy the properties (i)--(iv).   
    By $G_{ss}\cap a^{\Z} =\{1\}$ and the property (i)$_{ss}$, we see that $\varphi$ is injective, which is the property (i). 
    It follows from the choice of $a\in G$ and the property (ii)$_{ss}$ that
    the images of $\varphi(F_{t})$ under the two quotient homomorphisms of $\mathbf{G}$ by 
    $\mathbf{G}_{ss}$ and $\mathbf{Z}(\mathbf{G})_{0}$ are both Zariski-dense.
    Hence $\varphi(F_{t})$ is also Zariski-dense in $G$, which is the property (ii).
    The property (iii) follows immediately from the property (iii)$_{ss}$ and the choice of $U$, $a$.
    We also have the property (iv) obviously from the property (iv)$_{ss}$. Thus the claim is also 
    proved  in our setting.
\end{proof}

\begin{proof}[of Fact \ref{fact:Benoist_Cartan_cone} for reductive cases]
Applying Lemma \ref{lemma:Benoist-Schottky}, we take a compact neighborhood $U$ and 
a group homomorphism $\varphi\colon F_{t}\rightarrow G$ satisfying the properties (i)--(iv). 
It follows immediately from the properties (i)--(iv) 
that $\Gamma:=\varphi(F_{t})$ has the desired properties except that it is discrete in $G$.
Furthermore, we see from the properties (iii) and (iv) that $\mu(\Gamma)$ is discrete.
Hence $\Gamma$ is also discrete since the Cartan projection $\mu$ is a proper map.
This completes the proof of Fact \ref{fact:Benoist_Cartan_cone}.
\end{proof}

\section*{Acknowledgements.}
Our studies on Clifford--Klein forms have been influenced strongly by Professor Toshiyuki Kobayashi.
In particular, the first and second authors obtained their PhD degrees under the supervision of Professor Toshiyuki Kobayashi, and learned many things about discontinuous groups and representation theory.
Without his guidance, our studies would not have been possible.
The authors also would like to thank Fanny Kassel, Yoshiki Oshima, and Yosuke Morita for many helpful comments.
The first and third authors are supported by the Special Postdoctoral
Researcher Program at RIKEN.
The second author is supported by JSPS Grants-in-Aid for Scientific Research JP20K03589, JP20K14310, and JP22H0112.

\end{document}